\documentclass[11pt]{article}
\usepackage[top=1in, bottom=1in, left=1.25in, right=1.25in, marginparwidth=1in, marginparsep=0.1in]{geometry}

\usepackage{amsmath,amssymb,amsthm,xcolor,enumerate}
\usepackage{esint} %
\usepackage[unicode,breaklinks=true,colorlinks=true,citecolor = {magenta}]{hyperref}

\DeclareFontFamily{U}{mathx}{\hyphenchar\font45}
\DeclareFontShape{U}{mathx}{m}{n}{
      <5> <6> <7> <8> <9> <10>
      <10.95> <12> <14.4> <17.28> <20.74> <24.88>
      mathx10
      }{}
\DeclareSymbolFont{mathx}{U}{mathx}{m}{n}
\DeclareFontSubstitution{U}{mathx}{m}{n}
\DeclareMathAccent{\widecheck}{0}{mathx}{"71}
\numberwithin{equation}{section}
\newtheorem{thm}{Theorem}[section]

\newtheorem{lem}[thm]{Lemma}

\newtheorem{defn}[thm]{Definition}

\newtheorem{assum}[thm]{Assumption}
\theoremstyle{remark}
\newtheorem{remark}{Remark}[section]
\theoremstyle{definition}

\newcommand\al{\alpha}

\newcommand\ga{\gamma}
\newcommand\de{\delta}

\newcommand\ve{\varepsilon}
\newcommand\e {\varepsilon}  %

\newcommand\la{\lambda}

\newcommand\De{\Delta}

\newcommand{\R}{\mathbb{R}}

\newcommand{\ZZ}{\mathbb{Z}}
\newcommand{\NN}{\mathbb{N}}

\newcommand{\pd}{\partial}

\newcommand{\na}{\nabla}

\newcommand{\EQ}[1]{\begin{equation}\begin{split} #1 \end{split}\end{equation}}
\newcommand{\EQN}[1]{\begin{equation*}\begin{split} #1 \end{split}\end{equation*}}

\usepackage{theoremref}
\usepackage{mathrsfs}

\begin{document}
\title{Forward Discretely Self-Similar Solutions of the MHD Equations and the Viscoelastic Navier-Stokes Equations with Damping}

\author{Chen-Chih Lai}
\date{}

\maketitle
\begin{abstract}
In this paper, we prove the existence of forward discretely self-similar solutions to the MHD equations and the viscoelastic Navier-Stokes equations with damping with large weak $L^3$ initial data. The same proving techniques are also applied to construct self-similar solutions to the MHD equations and the viscoelastic Navier-Stokes equations with damping with large weak $L^3$ initial data. This approach is based on [Z. Bradshaw and T.-P. Tsai, Ann. Henri Poincar'{e}, vol. 18, no. 3, 1095-1119, 2017].



\end{abstract}
\section{Introduction}
The main purpose of this paper is to prove the existence of forward discretely self-similar (DSS) and self-similar (SS) weak solutions of both the MHD equations and the viscoelastic Navier-Stokes equations with damping. More precisely, we construct DSS local Leray weak solutions for DSS initial data with possibly large $L^3_w$-norm, and SS local Leray solutions for $(-1)$-homogeneous initial data in $L^3_w$. Our method follows from \cite{MR3611025} and is based on the a priori bounds \eqref{eq_1.15_mhd} and \eqref{eq_1.15_vNSEd}, and the Galerkin method. To begin with, we briefly introduce the MHD equations and the viscoelastic Navier-Stokes equations.
\subsection{The incompressible MHD equations}
In a magnetofluid, the interaction between the velocity field of the fluid and the magnetic field is governed by the coupling between the Navier-Stokes equations of fluid dynamics and Maxwell's equations of electromagnetism. The fundamental equations of magentohydrodynamics (MHD) is given by
\begin{equation}\label{MHD}
\setlength\arraycolsep{1.5pt}\def\arraystretch{1.2}
\left.\begin{array}{ll}
\pd_tv-\nu_0\De v+(v\cdot\na)v-(b\cdot\na)b+\na\pi&=0\ \\
\pd_tb-\eta_0\De b+(v\cdot\na)b-(b\cdot\na)v&=0\ \\  
~~~~~~~~~~~~~~~~~~~~~~~~~~~~~~\na\cdot v =\na\cdot b&=0\  
\end{array}\right\} \text{ in }\R^3\times(0,\infty),
\end{equation}
with initial data \[v|_{t=0}=v_0\ \text{ and }\ b|_{t=0}=b_0\ \text{ in }\R^3,\]
where $u:\R^3\times(0,\infty)\to\R^3$ is the fluid velocity, $b:\R^3\times(0,\infty)\to\R^3$ is the magnetic field, and $\pi:\R^3\times(0,\infty)\to\R$ represents the fluid pressure. The constants $\nu_0>0$ and $\eta_0>0$ are the kinetic viscosity and the magnetic resistivity, respectively. For simplicity, we assume $\nu_0=\eta_0=1$ throughout this paper.

We recall that the MHD equations \eqref{MHD} is invariant under the scaling 
\EQ{
v^\la(x,t)&=\la\, v(\la x,\la^2t),\ v_0^\la(x)=\la\,v_0(\la x),\\
b^\la(x,t)&=\la\, b(\la x,\la^2t),\ \, b_0^\la(x)=\la\,b_0(\la x),\\
\pi^\la(x,t)&=\la^2\pi(\la x,\la^2t).
}
We say that a solution $(v,b,\pi)$ of \eqref{MHD} is self-similar (SS) if it satisfies the scaling invariant $v^\la=v,\,b^\la=b$ and $\pi^\la=\pi$ for all $\la>0$. The initial data $v_0$ and $b_0$ are called self-similar if $v_0^\la=v_0$ and $b_0^\la=b_0$. On the other hand, if the scaling invariant only holds for a particular $\la>0$, we say $(v,b,\pi)$ is discretely self-similar with factor $\la>1$ ($\la$-DSS). Similarly, the initial data $v_0$ and $b_0$ are said to be $\la$-DSS if $v_0^\la=v_0$ and $b_0^\la=b_0$ for this $\la>1$.

On one hand, self-similar solutions of \eqref{MHD} have a stationary characteristic in that there exists an ansatz for $(v,b)$ in terms of time-independent profile $(u,a)$. That is, 
\begin{equation}\label{eq_1.6_mhd}
v(x,t)=\frac1{\sqrt{2t}}\,u\left(\frac{x}{\sqrt{2t}}\right),\ \ \ b(x)=\frac1{\sqrt{2t}}\,a\left(\frac{x}{\sqrt{2t}}\right),\ \ \ \pi(x,t)=\frac1{2t}\,p\left(\frac{x}{\sqrt{2t}}\right).
\end{equation}
The profile $(u,a)$ solves the stationary Leray system for the MHD equations
\begin{equation}\label{eq_1.7_mhd}
\setlength\arraycolsep{1.5pt}\def\arraystretch{1.2}
\left.\begin{array}{ll}
-\De u-u-y\cdot\na u+(u\cdot\na)u-(a\cdot\na)a+\na p&=0\ \\
-\De a-a-y\cdot\na a+(u\cdot\na)a-(a\cdot\na)u&=0\ \\  
~~~~~~~~~~~~~~~~~~~~~~~~~~~~~~~~~~~~~~~~~~\na\cdot u =\na\cdot b&=0  \ 
\end{array}\right\} \text{ in }\R^3\times\R,
\end{equation}
in the variable $y=x/\sqrt{2t}$. On the other hand, discretely self-similar solutions of \eqref{MHD} are determined by the behavior on the time intervals of the form $1\le t\le\la^2$. This leads us to consider the self-similar transform 
\begin{equation}
v(x,t)=\frac1{\sqrt{2t}}\,u(y,s),\ \ \ b(x,t)=\frac1{\sqrt{2t}}\,a(y,s),\ \ \ \pi(x,t)=\frac1{2t}\,p(y,s),
\end{equation}
where 
\begin{equation}\label{xtys}
y=\frac{x}{\sqrt{2t}},\ \ \ s=\log(\sqrt{2t}).
\end{equation}
Then $(u,a,p)$ solves the time-dependent Leray system for the MHD equations 
\begin{equation}\label{eq_1.10_mhd}
\setlength\arraycolsep{1.5pt}\def\arraystretch{1.2}
\left.\begin{array}{ll}
\pd_su-\De u-u-y\cdot\na u+(u\cdot\na)u-(a\cdot\na)a+\na p&=0\ \\
\pd_sa-\De a-a-y\cdot\na a+(u\cdot\na)a-(a\cdot\na)u&=0\ \\  
~~~~~~~~~~~~~~~~~~~~~~~~~~~~~~~~~~~~~~~~~~~~~~~~\na\cdot u =\na\cdot b&=0  \ 
\end{array}\right\} \text{ in }\R^3\times\R.
\end{equation}
Note that $(v,b,\pi)$ is $\la$-DSS if and only if $(u,a,p)$ is periodic in $s$ with the period $T=\log(\la)$.

Many significant contributions have been made concerning the existence of solutions to the MHD equations \eqref{MHD}. We list only some results related to our studies. First, Duvaut and Lions \cite{MR0346289} constructed a class of global weak solutions with finite energy and a class of local strong solutions. And the unique existence of mild solutions in BMO$^{-1}$ for small initial data has been obtained in Miao-Yuan-Zhang \cite{MR2313731}. In He-Xin \cite{MR2514362}, they also constructed a class of global unique forward SS solutions for small $(-1)$-homogeneous initial data belonging to some Besov space, or the Lorentz space or pseudo-measure space. Recently, Lin-Zhang-Zhou \cite{MR3487253} constructed a class of global smooth solution for large initial data assuming some constraints on the initial data on Fourier side.

\subsection{The incompressible viscoelastic Navier-Stokes equations with damping}
The Oldroyd-type models capture the rheological phenomena of both the fluid motions and the elastic features of non-Newtonian fluids. We study the simplest case in which the relaxation and retardation times are both infinite. More specifically, we consider the following system of equations for an incompressible, viscoelastic fluid: 
\begin{equation}\label{vNSE}
\setlength\arraycolsep{1.5pt}\def\arraystretch{1.2}
\left.\begin{array}{ll}
\pd_tv-\nu_0\De v+(v\cdot\na)v-\na\cdot({\bf F}{\bf F}^\top)+\na\pi&=0\ \\
\pd_t{\bf F}+(v\cdot\na){\bf F}-(\na v){\bf F}&=0\ \\  
~~~~~~~~~~~~~~~~~~~~~~~~~~~~~~~~~~~~~~~~~~~~~~~\na\cdot v &=0\  
\end{array}\right\} \text{ in }\R^3\times(0,\infty),
\end{equation}
with initial data \[v|_{t=0}=v_0\ \text{ and }\ {\bf F}|_{t=0}={\bf F}_0\ \text{ in }\R^3,\]
where $u:\R^3\times(0,\infty)\to\R^3$ is the velocity field, ${\bf F}:\R^3\times(0,\infty)\to\R^{3\times3}$ is the local deformation tensor of the fluid, and $\pi:\R^3\times(0,\infty)\to\R$ represents the pressure. The constant $\nu_0>0$ is the kinetic viscosity. Here $(\nabla\cdot({\bf F}{\bf F}^\top))_i=\pd_j({\bf F}_{ik}{\bf F}_{jk})$ and $(\na v)_{ij}=\pd_jv_i$. For convenience, we assume $\nu_0=1$ throughout this paper.

For the existence of weak solutions for the viscoelastic Navier-Stokes equations \eqref{vNSE}, it is well-known that short-time classical solutions and global existence of classical solutions for small initial data were established by Lin-Liu-Zhang \cite{MR2165379}. Later on, the authors \cite{MR2273974,MR2393434} proved the global existence of smooth solutions to \eqref{vNSE} in the case of near-equilibrium initial data. In \cite{MR2165379}, the authors added a damping term in the equation for ${\bf F}$ of the system \eqref{vNSE} to overcome the difficulty arises from the lack of a damping mechanism on ${\bf F}$. To be more precise, they introduced the following viscoelastic Navier-Stokes equations with damping as a way to approximate solutions of \eqref{vNSE}:
\begin{equation}\label{vNSEd0}
\setlength\arraycolsep{1.5pt}\def\arraystretch{1.2}
\left.\begin{array}{ll}
\pd_tv-\De v+(v\cdot\na)v-\na\cdot({\bf F}{\bf F}^\top)+\na\pi&=0\ \\
\pd_t{\bf F}-\mu\De {\bf F}+(v\cdot\na){\bf F}-(\na v){\bf F}&=0\ \\  
~~~~~~~~~~~~~~~~~~~~~~~~~~~~~~~~~~~~~~~~~~~~\na\cdot v &=0\  
\end{array}\right\} \text{ in }\R^3\times(0,\infty),
\end{equation}
for a damping parameter $\mu>0$. Note that if $\na\cdot{\bf F}=0$ at some instance of time, then $\na\cdot{\bf F}=0$ at all later times. In fact, by taking divergence of $\eqref{vNSEd0}_2$ and using $\eqref{vNSEd0}_3$, one have the following equation for $\na\cdot F$:
\[\pd_t(\na\cdot{\bf F})+(v\cdot\na)(\na\cdot{\bf F})=\mu\De(\na\cdot{\bf F}).\]
 Hence it is natural to assume 
\EQ{
\na\cdot{\bf F}=0.
}
Because the damping parameter $\mu$ plays no role in our construction of solutions, we set throughout this paper that \[\mu=1.\] Then, columnwisely, \eqref{vNSEd0} can be rewritten as 
\begin{equation}\label{vNSEd}
\setlength\arraycolsep{1.5pt}\def\arraystretch{1.2}
\left.\begin{array}{ll}
\pd_tv-\De v+(v\cdot\na)v-\underset{n=1}{\overset{3}\sum}(f_n\cdot\na)f_n+\na\pi&=0\ \\
\pd_tf_m-\De f_m+(v\cdot\na)f_m-(f_m\cdot\na)v&=0\ \\  
~~~~~~~~~~~~~~~~~~~~~~~~~~~~~~~~~~~~\na\cdot f_m=\na\cdot v &=0\  
\end{array}\right\} \text{ in }\R^3\times(0,\infty),\ m=1,2,3,
\end{equation}
where $f_m$ is the $m$-th column vector of ${\bf F}$.

Similar to the MHD equations, the viscoelastic equations with damping \eqref{vNSEd} is invariant under the scaling 
\EQ{
v^\la(x,t)&=\la\, v(\la x,\la^2t),\ v_0^\la(x)=\la\,v_0(\la x),\\
{\bf F}^\la(x,t)&=\la\, {\bf F}(\la x,\la^2t),\ \, {\bf F}_0^\la(x)=\la\,{\bf F}_0(\la x),\\
\pi^\la(x,t)&=\la^2\pi(\la x,\la^2t).
} 
We define SS and $\la$-DSS solution to \eqref{vNSEd} in the same manner as the ones we defined for the MHD equations. Self-similar solutions of \eqref{vNSEd} is determined by time-periodic profile $(u,{\bf F})$, where
\begin{equation}\label{eq_1.6_vNSEd}
v(x,t)=\frac1{\sqrt{2t}}\,u\left(\frac{x}{\sqrt{2t}}\right),\ \ \ {\bf F}(x)=\frac1{\sqrt{2t}}\,{\bf G}\left(\frac{x}{\sqrt{2t}}\right),\ \ \ \pi(x,t)=\frac1{2t}\,p\left(\frac{x}{\sqrt{2t}}\right),
\end{equation}
which satisfy the stationary Leray system for the viscoelastic Navier-Stokes equations with damping
\begin{equation}
\setlength\arraycolsep{1.5pt}\def\arraystretch{1.2}
\left.\begin{array}{ll}
-\De u-u-y\cdot\na u+(u\cdot\na)u-\underset{n=1}{\overset{3}\sum}(g_n\cdot\na)g_n+\na p&=0\ \\
-\De g_m-g_m-y\cdot\na g_m+(u\cdot\na)g_m-(g_m\cdot\na)u&=0\ \\  
~~~~~~~~~~~~~~~~~~~~~~~~~~~~~~~~~~~~~~~~~~~~~~~~\na\cdot u =\na\cdot g_m&=0  \ 
\end{array}\right\} \text{ in }\R^3\times\R,\ m=1,2,3,
\end{equation}
where $g_m$ is the $m$-th column vector of ${\bf G}$. For discretely self-similar solutions of \eqref{vNSEd}, we consider the self-similar transform 
\begin{equation}
v(x,t)=\frac1{\sqrt{2t}}\,u(y,s),\ \ \ {\bf F}(x,t)=\frac1{\sqrt{2t}}\,{\bf G}(y,s),\ \ \ \pi(x,t)=\frac1{2t}\,p(y,s),
\end{equation}
where $x,t,y,s$ satisfy \eqref{xtys}.
Then $(u,{\bf G},p)$ solves the time-dependent Leray system for the viscoelastic Navier-Stokes equations with damping
\begin{equation}\label{eq_1.10_vNSEd}
\setlength\arraycolsep{1.5pt}\def\arraystretch{1.2}
\left.\begin{array}{ll}
\pd_su-\De u-u-y\cdot\na u+(u\cdot\na)u-\underset{n=1}{\overset{3}\sum}(g_n\cdot\na)g_n+\na p&=0\ \\
\pd_sg_m-\De g_m-g_m-y\cdot\na g_m+(u\cdot\na)g_m-(g_m\cdot\na)u&=0\ \\  
~~~~~~~~~~~~~~~~~~~~~~~~~~~~~~~~~~~~~~~~~~~~~~~~~~~~~\na\cdot u =\na\cdot g_m&=0  \ 
\end{array}\right\} \text{ in }\R^3\times\R,\ m=1,2,3,
\end{equation}
where $g_m$ is the $m$-th column vector of ${\bf G}$. Note that $(v,{\bf F},\pi)$ is $\la$-DSS if and only if $(u,{\bf G},p)$ is periodic in $s$ with the period $T=\log(\la)$.

The authors \cite{MR2165379} mentioned that passing the limit of solutions to \eqref{vNSEd0} as $\mu\to0^+$ throughout standard weak convergence methods is not able to get weak solutions of \eqref{vNSE}. Despite of that, \eqref{vNSEd0} itself is still an interesting system, and there are a few of studies on this system. For instance, Lai-Lin-Wang \cite{MR3609234} established the existence of global forward SS classical solution to \eqref{vNSEd0} for locally H\"{o}lder continuous, $(-1)$-homogeneous initial data. For regularity issues, we refer the reader to \cite{MR3032986} and \cite{MR3626232}.

\subsection{Main results and Notation}
Our first goal is to extend the notion of weak solutions to the ones with a more general initial data. To this end, we recall the definition of local Leray weak solutions of the MHD equations \eqref{MHD}, which is consistent with the concept introduced by Lemari\'{e}-Rieusset \cite{MR1938147} on the Navier-Stokes equations. Here, for $1\le q<\infty$, let $L^q_{\textup{uloc}}$ denote the space of functions in $\R^3$ with 
\[\|f\|_{L^q_{\textup{uloc}}}:=\sup_{x_0\in\R^3}\|f\|_{L^q(B_1(x_0))}<\infty.\]
\begin{defn}[Local Leray solutions of the MHD equations]\thlabel{def_loc_leray_mhd}
A pair of vector fields $(v,b)$, where $v,b:\R^3\times[0,\infty)\to\R^3$ and $v,b\in L^2_{\textup{loc}}(\R^3\times[0,\infty))$, is called a local Leray solution to \eqref{MHD} with divergence-free initial data $v_0,\,b_0\in L_{\textup{uloc}}^2$ if 
\begin{enumerate}[$(i)$]
\item there exists $\pi\in L_{\textup{loc}}^{3/2}(\R^3\times[0,\infty))$ such that $(v,b,\pi)$ is a distributional solution to \eqref{MHD},
\item $($Locally finite energy$/$enstrophy$)$ for any $R>0$, $(v,b)$ satisfies 
\EQ{\label{lfee_mhd}
&\underset{0\le t<R^2}{\textup{esssup}}\ \underset{x_0\in\R^3}{\sup}\int_{B_R(x_0)}\frac12\left(|v(x,t)|^2+|b(x,t)|^2\right)dx\\
&~~~~~~~~~~~~~~~~~~~~~~~~+\underset{x_0\in\R^3}{\sup}\int_0^{R^2}\int_{B_R(x_0)}\left(|\na v(x,t)|^2+|\na b(x,t)|^2\right)dxdt<\infty,
}
\item $($Decay at spatial infinity$)$ for any $R>0$, $(v,b)$ satisfies 
\EQ{\label{dasi_mhd}
\lim_{|x_0|\to\infty}\int_0^{R^2}\int_{B_R(x_0)}\left(|v(x,t)|^2+|b(x,t)|^2\right)dxdt=0,
}
\item $($Convergence to initial data$)$ for all compact subsets $K$ of $\R^3$ we have $v(t)\to v_0$ and $b(t)\to b_0$ in $L^2(K)$ as $t\to0^+$,
\item $($Local energy inequality$)$ for all cylinders $Q$ compactly contained in $\R^3\times(0,\infty)$ and all nonnegetive $\phi\in C^\infty_0(Q)$, we have 
\EQ{\label{lei_mhd}
2\int\int\left(|\na v|^2+|\na b|^2\right)\phi\,dxdt\le&\int\int\left(|v|^2+|b|^2\right)\left(\pd_t\phi+\De\phi\right)dxdt\\&+\int\int\left(|v|^2+|b|^2+2\pi\right)(v\cdot\na\phi)dxdt\\&-2\int\int(v\cdot b)(b\cdot\na\phi)dxdt.
}
\end{enumerate}
\end{defn}

One of our goals in this paper is to prove the following existence theorem of a class of forward discretely self-similar solutions of the MHD equations \eqref{MHD}.
\begin{thm}\thlabel{thm_1.2_mhd}
Let $v_0$ and $b_0$ be divergence-free, $\la$-DSS vector fields for some $\la>1$ and satisfy 
\EQ{\label{eq_1.12_mhd}
\|v_0\|_{L^3_w(\R^3)}\le c_0,\ \ \ \|b_0\|_{L^3_w(\R^3)}\le c_0,
}
for some constant $c_0>0$. Then there exists a $\la$-DSS local Leray solution $(v,b)$ to \eqref{MHD}. Moreover, there exists $C_0=C_0(v_0,b_0)$ so that 
\[\|v(t)-e^{t\De}v_0\|_{L^2(\R^3)}\le C_0\,t^{1/4},\ \ \ \|b(t)-e^{t\De}b_0\|_{L^2(\R^3)}\le C_0\,t^{1/4}\] for any $t\in(0,\infty)$.
\end{thm}

Also, self-similar solutions of the MHD equations \eqref{MHD} can be constructed with $(-1)$-homogeneous initial data. Namely, we have
\begin{thm}\thlabel{thm_1.3_mhd}
Let $v_0$ and $b_0$ be divergence-free, $(-1)$-homogeneous and satisfy \eqref{eq_1.12_mhd} for some constant $c_0>0$. Then there exists a self-similar local Leray solution $(v,b)$ to \eqref{MHD}. In addition, there exists $C_0=C_0(v_0,b_0)$ such that \[\|v(t)-e^{t\De}v_0\|_{L^2(\R^3)}\le C_0\,t^{1/4},\ \ \ \|b(t)-e^{t\De}b_0\|_{L^2(\R^3)}\le C_0\,t^{1/4}\] for any $t\in(0,\infty)$.
\end{thm}

We would like to show similar results to \thref{thm_1.2_mhd} and \thref{thm_1.3_mhd} for the viscoelastic Navier-Stokes equations with damping \eqref{vNSEd}. For this purpose, we define analogous local Leray solutions to the viscoelastic Navier-Stokes equations with damping \eqref{vNSEd} as follows.
\begin{defn}[Local Leray solutions of the viscoelastic Navier-Stokes equations with damping]
A pair of a vector field and a tensor field $(v,{\bf F})$, where $u:\R^3\times(0,\infty)\to\R^3$, ${\bf F}:\R^3\times(0,\infty)\to\R^{3\times3}$ and $v,f_m\in L^2_{\textup{loc}}(\R^3\times[0,\infty))$ for $m=1,2,3$ with $f_m$ being the $m$-th column of ${\bf F}$, is called a local Leray solution to \eqref{vNSEd} with divergence-free initial data $v_0,\,{\bf F}_0\in L_{\textup{uloc}}^2$ if 
\begin{enumerate}[$(i)$]
\item there exists $\pi\in L_{\textup{loc}}^{3/2}(\R^3\times[0,\infty))$ such that $(v,{\bf F},\pi)$ is a distributional solution to \eqref{vNSEd},
\item $($Locally finite energy$/$enstrophy$)$ for any $R>0$, $(v,{\bf F})$ satisfies 
\EQ{\label{lfe_vNSEd}
&\underset{0\le t<R^2}{\textup{esssup}}\ \underset{x_0\in\R^3}{\sup}\int_{B_R(x_0)}\frac12\left(|v(x,t)|^2+|{\bf F}(x,t)|^2\right)dx\\
&~~~~~~~~~~~~~~~~~~~~~~+\underset{x_0\in\R^3}{\sup}\int_0^{R^2}\int_{B_R(x_0)}\left(|\na v(x,t)|^2+|\na{\bf F}(x,t)|^2\right)dxdt<\infty,
}
\item $($Decay at spatial infinity$)$ for any $R>0$, $(v,{\bf F})$ satisfies 
\EQ{\label{dasi_vNSEd}
\lim_{|x_0|\to\infty}\int_0^{R^2}\int_{B_R(x_0)}\left(|v(x,t)|^2+|{\bf F}(x,t)|^2\right)dxdt=0,
}
\item $($Convergence to initial data$)$ for all compact subsets $K$ of $\R^3$ we have $v(t)\to v_0$ and ${\bf F}(t)\to {\bf F}_0$ in $L^2(K)$ as $t\to0^+$,
\item $($Local energy inequality$)$ for all cylinders $Q$ compactly contained in $\R^3\times(0,\infty)$ and all nonnegative $\phi\in C^\infty_0(Q)$, we have 
\EQ{\label{lei_vNSEd}
2\int\int\left(|\na v|^2+|\na{\bf F}|^2\right)\phi\,dxdt\le&\int\int\left(|v|^2+|{\bf F}|^2\right)\left(\pd_t\phi+\De\phi\right)dxdt\\&+\int\int\left(|v|^2+|{\bf F}|^2+2\pi\right)(v\cdot\na\phi)dxdt\\&-2\,\underset{n=1}{\overset{3}\sum}\int\int(v\cdot f_n)(f_n\cdot\na\phi)dxdt.
}
\end{enumerate}
\end{defn}

The main theorems in this paper for the viscoelastic Navier-Stokes equations with damping can be stated as the following:
\begin{thm}\thlabel{thm_1.2_vNSEd}
Let $v_0$ and ${\bf F}_0$ be divergence-free, $\la$-DSS vector fields for some $\la>1$ and satisfy 
\EQ{\label{eq_1.12_vNSEd}
\|v_0\|_{L^3_w(\R^3)}\le c_0,\ \ \ \|{\bf F}_0\|_{L^3_w(\R^3)}\le c_0,
}
for some constant $c_0>0$. Then there exists a local Leray solution $(v,{\bf F})$ to \eqref{vNSEd} which is $\la$-DSS. Moreover, there exists $C_0=C_0(v_0,{\bf F}_0)$ so that 
\[\|v(t)-e^{t\De}v_0\|_{L^2(\R^3)}\le C_0\,t^{1/4},\ \ \ \|{\bf F}(t)-e^{t\De}{\bf F}_0\|_{L^2(\R^3)}\le C_0\,t^{1/4}\] for any $t\in(0,\infty)$.
\end{thm} 
\begin{thm}\thlabel{thm_1.3_vNSEd}
Let $v_0$ and ${\bf F}_0$ be divergence-free, $(-1)$-homogeneous and satisfy \eqref{eq_1.12_vNSEd} for some constant $c_0>0$. Then there exists a self-similar local Leray solution $(v,{\bf F})$ to \eqref{vNSEd}. In addition, there exists $C_0=C_0(v_0,{\bf F}_0)$ so that 
\[\|v(t)-e^{t\De}v_0\|_{L^2(\R^3)}\le C_0\,t^{1/4},\ \ \ \|{\bf F}(t)-e^{t\De}{\bf F}_0\|_{L^2(\R^3)}\le C_0\,t^{1/4}\] for any $t\in(0,\infty)$.
\end{thm}

\begin{remark}
The solutions obtained in \thref{thm_1.3_mhd} and \thref{thm_1.3_vNSEd} are actually infinitely smooth.
\end{remark}

The following a priori bounds are the keys to construct our desired solutions. For the MHD equations, if $(u,b)$ is a solution of \eqref{eq_1.10_mhd}, then the differences $U=u-U_0$ and $A=a-A_0$, where $U_0$ and $A_0$ are heat solutions, formally satisfy
\EQ{\label{eq_1.15_mhd}
&~~~~\int_0^T\int\left(|\na U|^2+|\na A|^2+\frac12\,|U|^2+\frac12\,|A|^2\right)\\
&=\int_0^T\int\left[(U\cdot\na)U\cdot U_0+(U\cdot\na)A\cdot A_0-(A\cdot\na)U\cdot A_0-(A\cdot\na)A\cdot U_0\right]\\
&~~~~-\int_0^T\int\left[\mathcal{R}_1(U_0,A_0)\cdot U+\mathcal{R}_2(U_0,A_0)\cdot A\right],
}
where $\mathcal{R}_1(U_0,A_0)$ and $\mathcal{R}_2(U_0,A_0)$ will be given in \eqref{eq_R1_R2}. Similarly, for the viscoelastic Navier-Stokes equations with damping, if $(u,g_1,g_2,g_3)$ is a solution of \eqref{eq_1.10_vNSEd}, then the differences $U=u-U_0$ and $G_m=g_m-G_{m,0}$, $m=1,2,3$, where $U_0$ and $G_{m,0}$ are heat solutions, formally obey 
\EQ{\label{eq_1.15_vNSEd}
&~~~~\int_0^T\int\left(|\na U|^2+\sum_{n=1}^3|\na G_n|^2+\frac12\,|U|^2+\frac12\,\sum_{n=1}^3|G_n|^2\right)\\
&=\int_0^T\int\left[(U\cdot\na)U\cdot U_0+\sum_{n=1}^3(U\cdot\na)G_n\cdot G_{n,0}-\sum_{n=1}^3(G_n\cdot\na)U\cdot G_{n,0}-\sum_{n=1}^3(G_n\cdot\na)G_n\cdot U_0\right]\\
&~~~~-\int_0^T\int\left[\mathcal{R}_3(U_0,G_1,G_2,G_3)\cdot U+\sum_{n=1}^3\mathcal{R}_4(U_0,G_{n,0})\cdot G_n\right],
}
where $\mathcal{R}_3(U_0,G_1,G_2,G_3)$ and $\mathcal{R}_4(U_0,G_{n,0})$ will be given in \eqref{eq_R3_R4}. Note that all cubic terms are either vanish or cancelled out in both \eqref{eq_1.15_mhd} and \eqref{eq_1.15_vNSEd}. To control the quadratic terms, we will choose a suitable cutoff to eliminate the possibly large local behavior of $U_0,\,A_0$ and $G_{m,0}$. See \thref{lem_2.5} for more details.

The rest of this paper is organized as follows. In Sect. 2, we recall some results in \cite{MR3611025} and construct a time-periodic solution to the Leray system for the MHD equations and the viscoelastic Navier-Stokes equations with damping. In Sect. 3, we recover discretely self-similar local Leray solutions for the MHD equations and the viscoelastic Navier-Stokes equations with damping from the solutions of the corresponding Leray systems obtained in Sect. 2. In Sect. 4, we prove the existence of self-similar local Leray solutions for the MHD equations and the viscoelastic Navier-Stokes equations with damping by constructing steady-state solutions to the Leray system for the MHD equations and the viscoelastic Navier-Stokes equations with damping, respectively.\\
\\
\emph{Notation.} We define the following function spaces 
\EQN{\mathcal{V}=\{f\in C^\infty_0(\R^3;\R^3):\na\cdot f=0\},\ X=\overline{\mathcal{V}}^{H^1_0(\R^3)},\ H=\overline{\mathcal{V}}^{L^2(\R^3)}.
}
Let $(\cdot,\cdot)$ be the $L^2(\R^3)$ inner product, and $\left<\cdot,\cdot\right>$ be the dual pairing of $H^1$ and its dual space $H^{-1}$, or that for $X$ and $X^*$. We denote \[\mathcal{D}_T=\left\{\varphi\in C^\infty(\R^3\times\R;\R^3):\begin{array}{l}\na \varphi=0,\,\varphi\text{ is periodic in $s$ with period $T$},\\ \textup{spt}(\varphi(\cdot,s))\text{ is compact in }\R^3\text{ for all }s\in[0,T)\end{array}\right\}.\]
We recall the Morrey space \[M^{p,\al}=\left\{f\in L^p_{\text{loc}}:\|f\|_{M^{p,\al}}:=\sup_{x\in\R^3,\,r>0}\left[r^{-\al}\int_{B_r(x)}|f|^p\right]^{1/p}<\infty\right\},\] and the weighted $L^2$ spaces 
\[L^2_{-k/2}=\left\{f\in L^2:\int_{\R^3}\frac{|f(x)|^2}{(1+|x|)^k}\,dx<\infty\right\}.\]

\section{The Time-Periodic Leray System}
\subsection{The time-periodic Leray system for the MHD equations}
In this subsection, we study the existence of time-periodic weak solutions to the Leray system for the MHD equations
\begin{equation}\label{leray_mhd}
\begin{split}
\setlength\arraycolsep{1.5pt}\def\arraystretch{1.2}
\left.\begin{array}{ll}
\pd_su-\De u-u-y\cdot\na u+(u\cdot\na)u-(a\cdot\na)a+\na p&=0\ \\
\pd_sa-\De a-a-y\cdot\na a+(u\cdot\na)a-(a\cdot\na)u&=0\ \\  
~~~~~~~~~~~~~~~~~~~~~~~~~~~~~~~~~~~~~~~~~~~~~~~~\na\cdot u =\na\cdot b&=0  \ 
\end{array}\right\}\ \text{ in }\R^3\times\R,\\
\lim_{|y_0|\to\infty}\int_{B_1(y_0)}\left(|u(y,s)-U_0(y,s)|^2+|a(y,s)-A_0(y,s)|^2\right)dy=0\ \text{ for all }s\in\R,\\
u(\cdot,s)=u(\cdot,s+T),\ a(\cdot,s)=a(\cdot,s+T)\ \text{ in }\R^3\text{ for all }s\in\R,
\end{split}
\end{equation}
for given $T$-periodic divergence-free vector fields $U_0$ and $A_0$. 

We first revisit the assumption for the background vector field $U_0$ and the corresponding results in \cite{MR3611025}.

\begin{assum}[\cite{MR3611025} Assumption 2.1]\thlabel{assum_2.1}
$U_0\in C^1(\R^4;\R^3)$ is periodic in $s$ with period $T>0$, divergence-free and satisfies
\[\pd_sU_0-\De U_0-U_0-y\cdot U_0=0,\]
\[U_0\in L^\infty(0,T;L^4\cap L^q(\R^3)),\]  
\[\na U_0\in L^2_{loc}(\R^4),\]
\[\pd_s U_0\in L^\infty(0,T;L^{6/5}_{\textup{loc}}(\R^3)),\]
and 
\[\sup_{s\in[0,T]}\|U_0(s)\|_{L^q(\R^3\setminus B_R)}\le\Theta(R),\]
for some $q\in(3,\infty]$ and $\Theta:[0,\infty)\to[0,\infty)$ such that $\Theta(R)\to0$ as $R\to\infty$.
\end{assum}

For notational simplicity, we define the linear differential operator $\mathcal{L}$ by 
\begin{equation}
\label{diff_op_L}\mathcal{L}W=\pd_sW-\De W-W-y\cdot\na W,
\end{equation} and so \[\left<\mathcal{L}W,\zeta\right>=(\pd_sW-W-y\cdot\na W,\zeta)+(\na W,\na\zeta)\] for all $\zeta\in C^1_0(\R^3)$.

\begin{lem}[\cite{MR3611025} Lemma 2.5]\thlabel{lem_2.5}
Fix $q\in(3,\infty]$ and suppose $U_0$ satisfies \thref{assum_2.1} for this $q$. Let $Z\in C^\infty(\R^3)$ with $0\le Z\le 1,\,Z(x)=1$ for $|x|>2$ and $Z(x)=0$ for $|x|<1$. For any $\de\in(0,1)$, there exists $R_0=R_0(U_0,\de)\ge1$ so that if we define 
$\xi(y)=Z\left(\frac{y}{R_0}\right)$, and  
\[
 w(y,s)=\int_{\R^3}\na_y\,\frac1{4\pi|y-z|}\,\na_z\xi(z)\cdot U_0(z,s)dz,
\]
then
\[
W(y,s)=\xi(y)U_0(y,s)+w(y,s)
\]
 has the following properties: locally continuously differentiable in $y$ and $s$, $T$-periodic, divergence-free, $U_0-W\in L^\infty(0,T;L^2(\R^3))\cap L^2(0,T;H^1(\R^3))$, and 
\begin{equation}\label{eq_2.8}
\|W\|_{L^\infty(0,T;L^q(\R^3))}\le\de,
\end{equation}
\begin{equation}\label{eq_2.9}
\|W\|_{L^\infty(0,T;L^4(\R^3))}\le c(R_0,U_0),
\end{equation}
and 
\begin{equation}\label{eq_2.10}
\|\mathcal{L}W\|_{L^\infty(0,T;H^{-1}(\R^3))}\le c(R_0,U_0),
\end{equation}
where $c(R_0,U_0)$ depends on $R_0$ and quantities associated with $U_0$ which are finite by \thref{assum_2.1}.
\end{lem}

\begin{lem}[\cite{MR3611025} Lemma 3.4]\thlabel{lem_3.4}
Suppose $v_0$ satisfies the assumption of \thref{thm_1.2_mhd} and let $x,t,y,s$ satisfy \eqref{xtys}. Then \[U_0(y,s):=\sqrt{2t}(e^{t\De}v_0)(x)\] satisfies \thref{assum_2.1} with $T=\log(\la)$ and any $q\in(3,\infty]$.
\end{lem}

Similar to the Navier-Stokes counterpart of time-periodic Leray system in \cite{MR3611025}, we define periodic weak solutions and suitable periodic weak solutions of \eqref{leray_mhd} as follows.

\begin{defn}[Periodic weak solution of Leray system for the MHD equations] Let $U_0$ and $A_0$ both satisfy \thref{assum_2.1}. A pair of vector fields $(u,a)$ is a periodic weak solution to \eqref{leray_mhd} if $\na\cdot u=\na\cdot b=0$, \[u-U_0,\,a-A_0\in L^\infty(0,T;L^2(\R^3))\cap L^2(0,T;H^1(\R^3)),\]
and 
\begin{equation}
\int_0^T\left[(u,\pd_s\varphi)-(\na u,\na \varphi)+(u+y\cdot\na u-u\cdot\na u+a\cdot\na a,\varphi)\right]ds=0,
\end{equation}
\begin{equation}
\int_0^T\left[(a,\pd_s\varphi)-(\na a,\na \varphi)+(a+y\cdot\na a-u\cdot\na a+a\cdot\na u,\varphi)\right]ds=0,
\end{equation}
holds for all $\varphi\in\mathcal{D}_T$.
\end{defn}

\begin{defn}[Suitable periodic weak solution of Leray system for the MHD equations] Let $U_0$ and $A_0$ both satisfy \thref{assum_2.1}. A triple $(u,a,p)$ is a suitable periodic weak solution to \eqref{leray_mhd} if $u,a,p$ are periodic in $s$ with period $T$, $(u,a)$ is a periodic weak solution to \eqref{leray_mhd}, $p\in L^{3/2}_{\textup{loc}}(\R^4)$, $(u,a,p)$ solves \eqref{leray_mhd} in the sense of distributions, and the local energy inequality holds: 
\EQ{\label{lei_mhd_leray}
\int_{\R^4}\left(\frac{|u|^2+|a|^2}2+|\na u|^2+|\na a|^2\right)\psi \,dyds\le&\int_{\R^4}\frac{|u|^2+|a|^2}2\left(\pd_s\psi+\De\psi\right)dyds\\&+\int_{\R^4}\left(\frac{|u|^2+|a|^2}2(u-y)+pu\right)\cdot\na\psi\,dyds\\&-\int_{\R^4}(u\cdot a)a\cdot\na\psi\,dyds,
} for all nonnegative $\psi\in C^\infty_0(\R^4)$.
\end{defn}

We are now ready to prove the existence of suitable periodic weak solutions of \eqref{leray_mhd}. Namely, we have
\begin{thm}[Existence of suitable periodic weak solutions to \eqref{leray_mhd}]\thlabel{thm_2.4_mhd}
Assume $U_0(y,s)$ and $A_0(y,s)$ both satisfy \thref{assum_2.1} with $q=10/3$. Then \eqref{leray_mhd} has a periodic suitable weak solution $(u,a,p)$ in $\R^4$ with period $T$.
\end{thm}
\begin{proof}
Fix $Z\in C^\infty(\R^3)$ with $0\le Z\le 1,\,Z(x)=1$ for $|x|>2$ and $Z(x)=0$ for $|x|<1$. 
Applying \thref{lem_2.5} with $\de=\frac14$, one can choose $R_0=R_0(U_0,A_0)\ge1$ such that letting $\xi(y)=Z\left(\frac{y}{R_0}\right)$ and setting
\begin{equation}\label{W_def}
W(y,s)=\xi(y)U_0(y,s)+w(y,s)
\end{equation}
and 
\begin{equation}\label{D_def}
D(y,s)=\xi(y)A_0(y,s)+d(y,s),
\end{equation}
where 
 \begin{equation}
 w(y,s)=\int_{\R^3}\na_y\,\frac1{4\pi|y-z|}\,\na_z\xi(z)\cdot U_0(z,s)dz
 \end{equation}
 and
  \begin{equation}
 d(y,s)=\int_{\R^3}\na_y\,\frac1{4\pi|y-z|}\,\na_z\xi(z)\cdot A_0(z,s)dz,
 \end{equation}
 both $W$ and $D$ satisfy the conclusion of \thref{lem_2.5}.
 
Using the differential operator $\mathcal{L}$ defined in \eqref{diff_op_L}, the Leray system \eqref{leray_mhd} can be written as 
\begin{equation}
\setlength\arraycolsep{1.5pt}\def\arraystretch{1.2}
\left\{\begin{array}{ll}
\mathcal{L}u+(u\cdot\na)u-(a\cdot\na)a+\na p&=0\\
\mathcal{L}a+(u\cdot\na)a-(a\cdot\na)u&=0\\
~~~~~~~~~~~~~~~~~~~~~~\na\cdot u=\na\cdot a&=0.
\end{array}\right.\end{equation}
We are looking for a solution of the form $u=U+W$ and $a=A+D$. Then $(U,A)$ must satisfy the perturbed Leray system for the MHD equations
\begin{equation}\label{ptb_leray_mhd}
\setlength\arraycolsep{1.5pt}\def\arraystretch{1.2}
\left\{\begin{array}{ll}
\mathcal{L}U+(W+U)\cdot\na U+U\cdot\na W-(D+A)\cdot\na A-A\cdot\na D+\na p&=-\mathcal{R}_1(W,D)\\
\mathcal{L}A+(W+U)\cdot\na A+U\cdot\na D-(D+A)\cdot\na U-A\cdot\na W&=-\mathcal{R}_2(W,D)\\
~~~~~~~~~~~~~~~~~~~~~~~~~~~~~~~~~~~~~~~~~~~~~~~~~~~~~~~~~~~~~~~~~\na\cdot U=\na\cdot A&=0,
\end{array}\right.
\end{equation}
where 
\begin{equation}\label{eq_R1_R2}
\setlength\arraycolsep{1.5pt}\def\arraystretch{1.2}
\left\{\begin{array}{l}
\mathcal{R}_1(W,D):=\mathcal{L}W+W\cdot\na W-D\cdot\na D\\
\mathcal{R}_2(W,D):=\mathcal{L}D+W\cdot\na D-D\cdot\na W.
\end{array}\right.
\end{equation}

We first solve the following mollified perturbed Leray system for the MHD equations for $(U^\ve,A^\ve,p^\ve)$ in $\R^3\times[0,T]$:
\begin{equation}\label{mdf_ptb_leray_mhd}
\setlength\arraycolsep{1.5pt}\def\arraystretch{1.2}
\left\{\begin{array}{ll}
&\mathcal{L}U^\ve+(W+(\eta_\ve*U^\ve))\cdot\na U^\ve+U^\ve\cdot\na W\\
&~~~~~~~~~~~~~~~~~~~~~~~~~~~-(D+(\eta_\ve*A^\ve))\cdot\na A^\ve-A^\ve\cdot\na D+\na p^\ve=-\mathcal{R}_1(W,D),\\
&\mathcal{L}A^\ve+(W+(\eta_\ve*U^\ve))\cdot\na A^\ve+U^\ve\cdot\na D\\
&~~~~~~~~~~~~~~~~~~~~~~~~~~~~~~~~~~~-(D+(\eta_\ve*A^\ve))\cdot\na U^\ve-A^\ve\cdot\na W=-\mathcal{R}_2(W,D),\\
&\,~~~~~~~~~~~~~~~~~~~~~~~~~~~~~~~~~~~~~~~~~~~~~~~~~~~~~~~~~~~~~\na\cdot U^\ve=\na\cdot A^\ve=0,
\end{array}\right.
\end{equation}
where $\eta_\ve(y)=\ve^{-3}\eta(y/\ve)$ for some fixed function $\eta\in C^\infty_0(\R^3)$ satisfying $\int_{\R^3}\eta dy=1$. The weak formulation of \eqref{mdf_ptb_leray_mhd} is 
\begin{equation}
\setlength\arraycolsep{1.5pt}\def\arraystretch{1.2}
\left\{\begin{array}{lll}
\frac{d}{ds}(U^\ve,f)&=&-(\na U^\ve,\na f)+(U^\ve+y\cdot\na U^\ve,f)-\left((\eta_\ve*U^\ve)\cdot\na U^\ve-(\eta_\ve*A^\ve)\cdot\na A^\ve,f\right)\\
&&-(W\cdot\na U^\ve+U^\ve\cdot\na W-D\cdot\na A^\ve-A^\ve\cdot\na D,f)-\left<\mathcal{R}_1(W,D),f\right>\\
\frac{d}{ds}(A^\ve,f)&=&-(\na A^\ve,\na f)+(A^\ve+y\cdot\na A^\ve,f)-\left((\eta_\ve*U^\ve)\cdot\na A^\ve-(\eta_\ve*A^\ve)\cdot\na U^\ve,f\right)\\
&&-(W\cdot\na A^\ve+U^\ve\cdot\na D-D\cdot\na U^\ve-A^\ve\cdot\na W,f)-\left<\mathcal{R}_2(W,D),f\right>
\end{array}\right.
\end{equation}
for all $f\in\mathcal{V}$ and a.e. $s\in(0,T)$.

~\\
{\bf Step 1: Construction of a solution to the mollified perturbed Leray system}

We use the Galerkin method to construct a solution of \eqref{mdf_ptb_leray_mhd}. Let $\{h_k\}_{k\in\NN}\subset\mathcal{V}$ be an orthonormal basis of $H$. Fixing a natural number $k$, we search for an approximation solution of the form $U^\ve_k(y,s)=\sum_{i=1}^k\mu^\ve_{ki}(s)h_i(y),\,A^\ve_k(y,s)=\sum_{i=1}^k\al^\ve_{ki}(s)h_i(y)$. We first prove the existence and an a priori estimate for $T$-periodic solutions $\mu^\ve_k=(\mu^\ve_{k1},\cdots,\mu^\ve_{kk}),\,\al^\ve_k=(\al^\ve_{k1},\cdots,\al^\ve_{kk})$ to the system of ODEs
\begin{equation}\label{galerkin_ode_mhd}
\setlength\arraycolsep{1.5pt}\def\arraystretch{1.2}
\left\{\begin{array}{ll}
\frac{d}{ds}\mu^\ve_{kj}&=\underset{i=1}{\overset{k}\sum}\mathscr{A}_{ij}\mu^\ve_{ki}+\underset{i=1}{\overset{k}\sum}\mathscr{B}_{ij}\al^\ve_{ki}+\underset{i,l=1}{\overset{k}\sum}\mathscr{C}^\ve_{ilj}\mu^\ve_{ki}\mu^\ve_{kl}-\underset{i,l=1}{\overset{k}\sum}\mathscr{C}^\ve_{ilj}\al^\ve_{ki}\al^\ve_{kl}+\mathscr{D}_j\\
\frac{d}{ds}\al^\ve_{kj}&=\underset{i=1}{\overset{k}\sum}\mathscr{E}_{ij}\mu^\ve_{ki}+\underset{i=1}{\overset{k}\sum}\mathscr{F}_{ij}\al^\ve_{ki}+\underset{i,l=1}{\overset{k}\sum}\mathscr{G}^\ve_{ilj}\mu^\ve_{ki}\al^\ve_{kl}+\mathscr{H}_j,
\end{array}\right.
\end{equation}
for $j=1,\cdots,k$, where 
\EQ{\label{galerkin_ode_coeff_mhd}
\mathscr{A}_{ij}&=-(\na h_i,\na h_j)+(h_i+y\cdot\na h_i,h_j)-(h_i\cdot\na W,h_j)-(W\cdot\na h_i,h_j),\\
\mathscr{B}_{ij}&=(h_i\cdot\na D,h_j)+(D\cdot\na h_i,h_j),\\
\mathscr{C}^\ve_{ilj}&=-((\eta_\ve*h_i)\cdot\na h_l,h_j),\\
\mathscr{D}_j&=-\left<\mathcal{R}_1(W,D),h_j\right>,\\
\mathscr{E}_{ij}&=-(h_i\cdot\na D,h_j)+(D\cdot\na h_i,h_j),\\
\mathscr{F}_{ij}&=-(\na h_i,\na h_j)+(h_i+y\cdot\na h_i,h_j)+(h_i\cdot\na W,h_j)-(W\cdot\na h_i,h_j),\\
\mathscr{G}^\ve_{ilj}&=-((\eta_\ve*h_i)\cdot\na h_l,h_j)+((\eta_\ve*h_l)\cdot\na h_i,h_j),\\
\mathscr{H}_j&=-\left<\mathcal{R}_2(W,D),h_j\right>.
}
Fix $k\in\NN$. For any $U^0,A^0\in\textup{span}(h_1,\cdots,h_k)$, there exist $\mu^\ve_{kj},\al^\ve_{kj}\in H^1(0,\tilde{T})$, $j=1,\cdots,k$, that uniquely solve \eqref{galerkin_ode_mhd} with initial data $\mu^\ve_{kj}(0)=(U^0,h_j),\,\al^\ve_{kj}(0)=(A^0,h_j)$, $j=1,\cdots,k$, for some $0<\tilde{T}\le T$. 

We show that $\tilde{T}=T$. To this end, we first derive \EQ{\label{eq_2.27_mhd}
&~~~~\frac12\,\frac{d}{ds}\left(\|U^\ve_k\|_{L^2}^2+\|A^\ve_k\|_{L^2}^2\right)+\frac12\left(\|U^\ve_k\|_{L^2}^2+\|A^\ve_k\|_{L^2}^2\right)+\left(\|\na U^\ve_k\|_{L^2}^2+\|\na  A^\ve_k\|_{L^2}^2\right)\\
&=-(U^\ve_k\cdot\na W-D\cdot\na A^\ve_k-A^\ve_k\cdot\na D,U^\ve_k)-(U^\ve_k\cdot\na D-D\cdot\na U^\ve_k-A^\ve_k\cdot\na W,A^\ve_k)\\
&~~~-\left<\mathcal{R}_1(W,D),U^\ve_k\right>-\left<\mathcal{R}_2(W,D),A^\ve_k\right>
,}
by multiplying the $j$-th equation of $\eqref{galerkin_ode_mhd}_1$ by $\mu^\ve_{kj}$, and multiplying the $j$-th equation of $\eqref{galerkin_ode_mhd}_2$ by $\al^\ve_{kj}$, and then sum up all $2k$ equations. In the derivation, notice that $((\eta_\ve*U^\ve)\cdot\na U^\ve,U^\ve),\,(W\cdot\na U^\ve,U^\ve),\,((\eta_\ve*U^\ve)\cdot\na A^\ve,A^\ve)$ and $(W\cdot\na A^\ve,A^\ve)$ vanish, and $((\eta_\ve*A^\ve)\cdot\na A^\ve,U^\ve)$ and $((\eta_\ve*A^\ve)\cdot\na U^\ve,A^\ve)$ are cancelled each other; thus these terms don't show up in \eqref{eq_2.27_mhd}. Using \thref{lem_2.5} with $\de=\frac14$, we get 
\EQ{\label{etm_2.28_mhd}
&~~~~\left|-(U^\ve_k\cdot\na W-D\cdot\na A^\ve_k-A^\ve_k\cdot\na D,U^\ve_k)-(U^\ve_k\cdot\na D-D\cdot\na U^\ve_k-A^\ve_k\cdot\na W,A^\ve_k)\right|\\
&\le\frac38\left(\|U^\ve_k\|_{H^1}^2+\|A^\ve_k\|_{H^1}^2\right),
}
and
\EQ{\label{etm_2.29_mhd}
\left|-\left<\mathcal{R}_1(W,D),U^\ve_k\right>-\left<\mathcal{R}_2(W,D),A^\ve_k\right>\right|\le C_2+\frac3{32}\left(\|U^\ve_k\|_{H^1}^2+\|A^\ve_k\|_{H^1}^2\right),
}
where $C_2=8\left(\|\mathcal{L}W\|_{H^{-1}}^2+\|\mathcal{L}D\|_{H^{-1}}^2+(\|W\|_{L^4}^2+\|D\|_{L^4}^2)^2\right)$ is independent of $s,T,k$ and $\varepsilon$.

Using the estimates \eqref{etm_2.28_mhd} and \eqref{etm_2.29_mhd}, we obtain from \eqref{eq_2.27_mhd} the differential inequality
\EQ{\label{eq_2.30_mhd}
\frac{d}{ds}\left(\|U^\ve_k\|_{L^2}^2+\|A^\ve_k\|_{L^2}^2\right)+\frac1{16}\left(\|U^\ve_k\|_{L^2}^2+\|A^\ve_k\|_{L^2}^2\right)+\frac1{16}\left(\|\na U^\ve_k\|_{L^2}^2+\|\na  A^\ve_k\|_{L^2}^2\right)\le C_2.
}
Applying the Gronwall inequality, we get 
\EQ{\label{eq_2.31_mhd}
e^{s/{16}}\left(\|U^\ve_k\|_{L^2}^2+\|A^\ve_k\|_{L^2}^2\right)\le&~\left(\|U^0\|_{L^2}^2+\|A^0\|_{L^2}^2\right)+\int_0^{\tilde{T}}e^{\tau/{16}}C_2\,d\tau\\
\le&~\left(\|U^0\|_{L^2}^2+\|A^0\|_{L^2}^2\right)+e^{T/{16}}C_2T
}
for all $s\in[0,\tilde{T}]$. Since the right-hand side is finite, $\tilde{T}$ is not a blow-up time and we conclude that $\tilde{T}=T$.

Choosing $\rho=\frac{C_2T}{1-e^{-T/{16}}}>0$ (independent of $k$), \eqref{eq_2.31_mhd} implies that \[\left(\|U^\ve_k\|_{L^2}^2+\|A^\ve_k\|_{L^2}^2\right)^\frac12\le\rho\] if $\left(\|U^0\|_{L^2}^2+\|A^0\|_{L^2}^2\right)^\frac12\le\rho$. Define $\mathcal{T}:B_\rho^{2k}\to B_\rho^{2k}$ by $\mathcal{T}(\mu^\ve_k(0),\al^\ve_k(0))=(\mu^\ve_k(T),\al^\ve_k(T))$, where $B_\rho^{2k}$ is the closed ball in $\R^{2k}$ of radius $\rho$ and centered at the origin. Note that the map $\mathcal{T}$ is continuous by the continuous dependence on initial conditions of the solution of ODEs. Thus, it has a fixed point by the Brouwer fixed point theorem, i.e., there exist $(\mu^\ve_k(0),\al^\ve_k(0))\in B_\rho^{2k}$ such that $(\mu^\ve_k(0),\al^\ve_k(0))=(\mu^\ve_k(T),\al^\ve_k(T))$. Let $U^0=\sum_{i=1}^k\mu^\ve_{ki}(0)h_i$ and $A^0=\sum_{i=1}^k\al^\ve_{ki}(0)h_i$. Then $U^0,A^0\in\textup{span}(h_1,\cdots,h_k)$ and $U^0=U^\ve_k(T),A^0=A^\ve_k(T)$.

With the choice of $U^0$ and $A^0$ we have $\left(\|U^\ve_k(s)\|_{L^2}^2+\|A^\ve_k(s)\|_{L^2}^2\right)^\frac12\le\rho$ for all $s\in[0,T]$. Hence
\begin{equation}
\left(\|U^\ve_k\|_{L^\infty(0,T;L^2(\R^3))}^2+\|A^\ve_k\|_{L^\infty(0,T;L^2(\R^3))}^2\right)^\frac12\le\rho.
\end{equation}
Moreover, by integrating \eqref{eq_2.30_mhd} in $s\in[0,T]$ and using $U^\ve_k(0)=U^\ve_k(T),\,A^\ve_k(0)=A^\ve_k(T)$, we get 
\begin{equation}
\frac1{16}\int_0^T\left(\|U^\ve_k(s)\|_{H^1}^2+\|A^\ve_k(s)\|_{H^1}^2\right)ds\le C_2T.
\end{equation}
Therefore, 
\begin{equation}\label{eq_2.26_mhd}
\|U^\ve_k\|_{L^\infty(0,T;L^2(\R^3))}+\|A^\ve_k\|_{L^\infty(0,T;L^2(\R^3))}+\|U^\ve_k\|_{L^2(0,T;H^1(\R^3))}+\|A^\ve_k\|_{L^2(0,T;H^1(\R^3))}\le C,
\end{equation}
where $C=\sqrt{4(\rho^2+16C_2T)}$ is independent of both $\ve$ and $k$.

Using the uniform bounded sequences $\{U^\ve_k\}_{k\in\NN}$ and $\{A^\ve_k\}_{k\in\NN}$, and a standard limiting process, we get, for all $\ve>0$, two $T$-periodic vector fields $U^\ve,A^\ve\in L^2(0,T;H^1_0(\R^3))$ (both have $\ve$-independent $L^\infty L^2$ and $L^2H^1$ bounds), a subsequence of $\{U^\ve_k\}_{k\in\NN}$, and a subsequence of $\{A^\ve_k\}_{k\in\NN}$ (still denoted by $U^\ve_k$ and $A^\ve_k$, respectively) so that
\EQ{
&U^\ve_k\rightharpoonup U^\ve,\ A^\ve_k\rightharpoonup A^\ve\ \text{ weakly in $L^2(0,T;X)$},\\
&U^\ve_k\to U^\ve,\ A^\ve_k\to A^\ve\ \text{ strongly in $L^2(0,T;L^2(K))$ for all compact sets $K\subset\R^3$},\\
&U^\ve_k(s)\rightharpoonup U^\ve(s),\ A^\ve_k(s)\rightharpoonup A^\ve(s)\ \text{ weakly in $L^2$ for all $s\in[0,T]$}.
}
The weak convergence guarantees that $U^\ve(0)=U^\ve(T)$ and $A^\ve(0)=A^\ve(T)$. Moreover, the pair $(U^\ve,A^\ve)$ is a periodic weak solution of the mollified perturbed Leray system \eqref{mdf_ptb_leray_mhd}.

~\\
{\bf Step 2: A priori estimate of the pressure in the mollified perturbed Leray system}

Note that $\na\cdot\mathcal{L}V=0$ if $\na\cdot V=0$. Therefore, by taking the divergence of $\eqref{mdf_ptb_leray_mhd}_1$, we obtain
\EQ{\label{eq_2.33_mhd}
-\De p^\ve=&\sum_{i,j=1}^k\pd_i\pd_j\left[(\eta_\ve*U^\ve_i)U^\ve_j+W_iU^\ve_j+U^\ve_i W_j+W_iW_j\right.\\
&~~~~~~~~~~~~~\left.-(\eta_\ve*A^\ve_i)A^\ve_j-D_iA^\ve_j-A^\ve_i D_j-D_iD_j\right].
}
Let 
\EQ{
\tilde{p}^\ve=&\sum_{i,j=1}^kR_iR_j\left[(\eta_\ve*U^\ve_i)U^\ve_j+W_iU^\ve_j+U^\ve_i W_j+W_iW_j\right.\\&~~~~~~~~~~~~~~~\left.-(\eta_\ve*A^\ve_i)A^\ve_j+D_iA^\ve_j+A^\ve_i D_j+D_iD_j\right],
}
where $R_i$ denote the Riesz transforms. Note that $\tilde{p}^\ve$ also satisfies \eqref{eq_2.33_mhd}. We will show that $p^\ve=\tilde{p}^\ve$ up to an additive constant by proving $\na(p^\ve-\tilde{p}^\ve)=0$. 

Let $V^\ve(x,t)=(2t)^{-1/2}U^\ve(y,s),\,\pi^\ve(x,t)=(2t)^{-1}p^\ve(y,s)$ and $\mathcal{F}^\ve(x,t)=(\mathcal{F}_1+\mathcal{F}^\ve_2)(x,t)$ where 
\begin{equation}
\mathcal{F}_1(x,t):=-\frac1{(2t)^{3/2}}(\mathcal{L}W)(y,s),
\end{equation}
\EQ{
\mathcal{F}^\ve_2(x,t):=-\frac1{(2t)^{3/2}}&\left[W\cdot\na U^\ve+U^\ve\cdot\na W+(\eta_\ve*U^\ve)\cdot\na U^\ve+W\cdot\na W\right.\\&\left.\,-D\cdot\na A^\ve-A^\ve\cdot\na D-(\eta_\ve*A^\ve)\cdot\na A^\ve-D\cdot\na D\right](y,s),
}
and $y=x/\sqrt{2t}$ and $s=\log(\sqrt{2t})$. Hence, $\mathcal{F}^\ve\in L^\infty(1,\la^2;H^{-1}(\R^3))$, $(V^\ve,\pi)$ solves the non-stationary Stokes system on $\R^3\times[1,\la^2]$ with force $\mathcal{F}^\ve$ defined by $\eqref{mdf_ptb_leray_mhd}_1$, and $V^\ve$ is in the energy class. According to the uniqueness of the solution to the forced, non-stationary Stokes system on $\R^3\times[1,\la^2]$, we can conclude that $\na\pi^\ve=\na\tilde{\pi}^\ve$ where $\tilde{\pi}^\ve=(2t)^{-1}\tilde{p}^\ve$. Therefore $\na(p^\ve-\tilde{p}^\ve)=0$.

At this stage, we may replace $p^\ve$ by $\tilde{p}^\ve$. Recall that the Riesz transforms $R_i\phi(x)=\lim_{\e\to0^+}\int_{|x-y|>\e}K_i(x-y)\phi$ are Calder\'on-Zygmund operators since $K_i(x)=\frac{x_i}{|x|^{n+1}}$ are Calder\'on-Zygmund kernels. Applying the Calder\'on-Zygmund theory, we get
\EQN{
\|p^\ve(s)\|_{L^{5/3}}\le&~C\left\|\left[(\eta_\ve*U^\ve_i)U^\ve_j+W_iU^\ve_j+U^\ve_i W_j+W_iW_j\right.\right.\\&~~~~~\left.\left.-(\eta_\ve*A^\ve_i)A^\ve_j+D_iA^\ve_j+A^\ve_i D_j+D_iD_j\right](s)\right\|_{L^{5/3}}\\\le&~C\left(\|U^\ve(s)\|_{L^{10/3}}^2+\|A^\ve(s)\|_{L^{10/3}}^2+\|W(s)\|_{L^{10/3}}^2+\|D(s)\|_{L^{10/3}}^2\right).
}
Hence we obtain the following a priori bound for $p^\ve$:
\EQ{\label{eq_2.37_mhd}
\|p^\ve\|_{L^{5/3}(\R^3\times[0,T])}\le&~C\left(\|U^\ve\|_{L^{10/3}(\R^3\times[0,T])}^2+\|A^\ve\|_{L^{10/3}(\R^3\times[0,T])}^2\right.\\&~~~~~\left.+\|W\|_{L^{10/3}(\R^3\times[0,T])}^2+\|D\|_{L^{10/3}(\R^3\times[0,T])}^2\right).
}
Recall that the sequences $\{U^\ve\}_{\ve>0}$ and $\{A^\ve\}_{\ve>0}$ are both bounded in $L^\infty L^2$ and $L^2H^1$ norms. So 
\EQ{\label{bound_U^ve_mhd}
\|U^\ve\|_{L^{10/3}(\R^3\times[0,T])}=\left\|\|U^\ve\|_{L_y^{10/3}}\right\|_{L_s^{10/3}}\le&~\left\|\|U^\ve\|_{L_y^2}^{\frac25}\|U^\ve\|_{L_y^6}^{\frac35}\right\|_{L_s^{10/3}}\\
\le&~\|U^\ve\|_{L^\infty(0,T;L^2(\R^3))}^{\frac25}\left\|\|U^\ve\|_{L_y^6}^{\frac35}\right\|_{L_s^{10/3}}\\
\le&~\|U^\ve\|_{L^\infty(0,T;L^2(\R^3))}^{\frac25}\|U^\ve\|_{L^2(0,T;L^6(\R^3))}^{\frac35}\\
\lesssim&~\|U^\ve\|_{L^\infty(0,T;L^2(\R^3))}^{\frac25}\|U^\ve\|_{L^2(0,T;H^1(\R^3))}^{\frac35}\\
\le&~C,
}
where $C$ is some constant independent of $\ve$. Similarly, we also obtain 
\EQ{\label{bound_A^ve_mhd}
\|A^\ve\|_{L^{10/3}(\R^3\times[0,T])}\le C.
}
In addition, because we are applying \thref{lem_2.5} with $q=\frac{10}3$ and $\de=\frac14$, we have $\|W\|_{L^\infty(0,T;L^{10/3}(\R^3))}\le\frac14$ and $\|D\|_{L^\infty(0,T;L^{10/3}(\R^3))}\le\frac14$. Thus, we have the esitmates
\EQ{\label{bound_WD}
\|W\|_{L^{10/3}(\R^3\times[0,T])}\le \frac14\,T^{10/3}\ \ \ \text{ and }\ \ \ \|D\|_{L^{10/3}(\R^3\times[0,T])}\le \frac14\,T^{10/3}.
}

Using the bounds \eqref{bound_U^ve_mhd}-\eqref{bound_WD}, \eqref{eq_2.37_mhd} implies that $\{p^\ve\}_{\ve>0}$ is a bounded sequence in $L^{5/3}(\R^3\times[0,T])$.

~\\
{\bf Step 3: Convergence to a suitable periodic weak solution to \eqref{leray_mhd}}

Since the sequences $\{U^\ve\}_{\ve>0}$ and $\{A^\ve\}_{\ve>0}$ are both bounded in $L^\infty(0,T;L^2(\R^3))$- and $L^2(0,T;H^1(\R^3))$- norms, there exist $U,A\in L^\infty(0,T;L^2(\R^3))\cap L^2(0,T;H^1_0(\R^3))$ and two sequences $\{U^{\ve_k}\}_{k\in\NN},\{A^{\ve_k}\}_{k\in\NN}$ such that
\EQ{\label{U_A_conv}
&U^{\ve_k}\rightharpoonup U,\ A^{\ve_k}\rightharpoonup A\ \ \ \text{ weakly in $L^2(0,T;X)$},\\
&U^{\ve_k}\to U,\ A^{\ve_k}\to A\ \ \ \text{ strongly in $L^2(0,T;L^2(K))$ for all compact sets $K\subset\R^3$},\\
&U^{\ve_k}(s)\rightharpoonup U(s),\ A^{\ve_k}(s)\rightharpoonup A(s)\ \ \ \text{ weakly in $L^2$ for all $s\in[0,T]$},
}
as $\ve_k\to0$.

On the other hand, since $\{p^{\ve_k}\}_{k\in\NN}$ is a bounded sequence in $L^{5/3}(\R^3\times[0,T])$, we have that 
\EQ{
p^{\ve_k}\rightharpoonup p\ \text{ weakly in $L^{5/3}(\R^3\times[0,T])$,}
}
for some $p\in L^{5/3}(\R^3\times[0,T])$. Let $u=U+W$ and $a=A+D$. The above convergences are enough to ensure that the triple $(u,a,p)$ solves \eqref{leray_mhd} in the sense of distributions.

It remains to check that $(u,a,p)$ satisfies the local energy inequality \eqref{lei_mhd_leray}. Note that $(u^{\ve_k},a^{\ve_k},p^{\ve_k})$, where $u^{\ve_k}=U^{\ve_k}+W$ and $a^{\ve_k}=A^{\ve_k}+D$, satisfies
\begin{equation}\label{eq_u^ve_a^ve}
\setlength\arraycolsep{1.5pt}\def\arraystretch{1.2}
\left\{\begin{array}{ll}
&\mathcal{L}u^{\ve_k}+W\cdot\na u^{\ve_k}+(\eta_{\ve_k}*U^{\ve_k})\cdot\na U^{\ve_k}+U^{\ve_k}\cdot\na W\\
&~~~~~~~~~~~~~~~~~~~~~~~~~~~~-D\cdot\na a^{\ve_k}-(\eta_{\ve_k}*A^{\ve_k})\cdot\na A^{\ve_k}-A^{\ve_k}\cdot\na D+\na p^{\ve_k}=0,\\
&\mathcal{L}a^{\ve_k}+W\cdot\na a^{\ve_k}+(\eta_{\ve_k}*U^{\ve_k})\cdot\na A^{\ve_k}+U^{\ve_k}\cdot\na D\\
&~~~~~~~~~~~~~~~~~~~~~~~~~~~~~~~~~~~~~-D\cdot\na u^{\ve_k}-(\eta_{\ve_k}*A^{\ve_k})\cdot\na U^{\ve_k}-A^{\ve_k}\cdot\na W=0.
\end{array}\right.
\end{equation}

Testing $\eqref{eq_u^ve_a^ve}_1$ and $\eqref{eq_u^ve_a^ve}_2$ with $u^{\ve_k}\psi$ and $a^{\ve_k}\psi$, respectively, where $0\le\psi\in C^\infty_0(\R^4)$ and adding them together, we get
\EQ{\label{A_51_mhd}
&~~~~\int_{\R^4}\left(\frac{|u^{\ve_k}|^2+|a^{\ve_k}|^2}2+|\na u^{\ve_k}|^2+|\na a^{\ve_k}|^2\right)\psi \,dyds\\&=\int_{\R^4}\frac{|u^{\ve_k}|^2+|a^{\ve_k}|^2}2\left(\pd_s\psi+\De\psi\right)dyds+\int_{\R^4}\frac{|u^{\ve_k}|^2+|a^{\ve_k}|^2}2\,(W-y)\cdot\na\psi\,dyds\\
&~~~+\int_{\R^4}\left(\frac{|U^{\ve_k}|^2+2(U^{\ve_k}\cdot W)+|A^{\ve_k}|^2+2(A^{\ve_k}\cdot D)}2\,\left(\eta_{\ve_k}*U^{\ve_k}\right)\right.\\
&~~~~~~~~~~~~~~~~~~~~~~~~~~~~~~~~~~~~~~~~~~~~~~~~~~~~~~~~~~~~~~~~~~~\left.+\frac{|W|^2+|D|^2}2\,U^{\ve_k}\right)\cdot\na\psi\,dyds\\
&~~~+\int_{\R^4}p^{\ve_k} u^{\ve_k}\cdot\na\psi\,dyds\\
&~~~-\int_{\R^4}\left((u^{\ve_k}\cdot a^{\ve_k})D+(U^{\ve_k}\cdot A^{\ve_k})(\eta_{\ve_k}*A^{\ve_k})+(U^{\ve_k}\cdot D)A^{\ve_k}+(W\cdot A^{\ve_k})A^{\ve_k}\right.\\
&~~~~~~~~~~~~~~~~~~~~~~~~~~~~~~~~~~~~~~~~~~~~~~~~~~~~~~~~~~~~~~~~~~~~~~~~~~\left.+(W\cdot D)A^{\ve_k}\right)\cdot\na\psi\,dyds\\
&~~~+\int_{\R^4}\left((\eta_{\ve_k}*U^{\ve_k})-U^{\ve_k}\right)\cdot(\na W\cdot U^{\ve_k}+\na D\cdot A^{\ve_k})\psi\,dyds\\
&~~~+\int_{\R^4}((\eta_{\ve_k}*A^{\ve_k})-A^{\ve_k})\cdot(\na U^{\ve_k}\cdot D+\na A^{\ve_k}\cdot W)\psi\,dyds.
}
Let $\mathcal{K}$ be a compact subset of $\R^4$. We have 
\EQN{
\|(\eta_{\ve_k}*U^{\ve_k})-U\|_{L^2(\mathcal{K})}\le&~\|(\eta_{\ve_k}*U^{\ve_k})-(\eta_{\ve_k} *U)\|_{L^2(\mathcal{K})}+\|(\eta_{\ve_k}*U)-U\|_{L^2(\mathcal{K})}\\\le&~\|U^{\ve_k}-U\|_{L^2(\mathcal{K})}+\|\eta_{\ve_k}*U-U\|_{L^2(\mathcal{K})}.
}
Since $\|(\eta_{\ve_k}*U)(s)-U(s)\|_{L_y^2}\le \|(\eta_{\ve_k}*U)(s)\|_{L_y^2}+\|U(s)\|_{L_y^2}\le2\|U(s)\|_{L_y^2}\in L_s^2(I)$ for all compact interval $I$, dominated convergence theorem implies that $\|(\eta_{\ve_k}*U)-U\|_{L^2(\mathcal{K})}\to0$ as ${\ve_k}\to0$. Together with the fact that $U^{\ve_k}\to U$ in $L^2(0,T;L^2(K))$ for all compact sets $K\subset\R^3$, we conclude that 
\EQ{\label{eta_U_conv_mhd}
\|(\eta_{\ve_k}*U^{\ve_k})-U\|_{L^2(\mathcal{K})}\to0\ \text{ as }\ \ve_k\to0\ \ \ \text{ for all compact }\mathcal{K}\subset\R^4.
} 
Similarly, we have
\EQ{\label{eta_A_conv_mhd}
\|(\eta_{\ve_k}*A^{\ve_k})-A\|_{L^2(\mathcal{K})}\to0\ \text{ as }\ \ve_k\to0\ \ \ \text{ for all compact }\mathcal{K}\subset\R^4.
}
In addition, the sequence $\{u^\ve\}_{\ve>0}$ is bounded in $L^{10/3}(\R^3\times[0,T])$ since it is bounded in $L^2(0,T;L^6(\R^3))$ and $L^\infty(0,T;L^2(\R^3))$. According to the well-known fact mentioned in the Appendix of \cite{MR673830}, 
\EQ{\label{u_strong_5/2_mhd}
u^{\ve_k}\to u\ \ \ \text{ strongly in }L^{5/2}(\mathcal{K})\ \text{ as }\ve_k\to0.
}

Combining \eqref{eta_U_conv_mhd}-\eqref{u_strong_5/2_mhd} and the convergences in \eqref{U_A_conv} with the facts that $W,\,D$ are locally differentiable and that the support of $\psi$ is compact, each term on the right hand side of \eqref{A_51_mhd} converges to the corresponding term involving $u,\,U,\,a,\,A$ and $p$. On the other hand, $\int\na |u^{\ve_k}|^2dyds$ and $\int|\na a^{\ve_k}|^2dyds$ are lower-semicontinuous as $\ve_k\to0$. This proves \eqref{lei_mhd_leray} and completes the proof of \thref{thm_2.4_mhd}.

\end{proof}

\subsection{The time-periodic Leray system for the viscoelastic Navier-Stokes equations with damping}
In this subsection, we follow the same approach as in Sect. 2.1 to construct a periodic weak solution to the Leray system for the viscoelastic Navier-Stokes equations with damping
\begin{equation}\label{leray_vNSEd}
\begin{split}
\setlength\arraycolsep{1.5pt}\def\arraystretch{1.2}
\left.\begin{array}{ll}
\pd_su-\De u-u-y\cdot\na u+(u\cdot\na)u-\underset{n=1}{\overset{3}\sum}(g_n\cdot\na)g_n+\na p&=0\ \\
\pd_sg_m-\De g_m-g_m-y\cdot\na g_m+(u\cdot\na)g_m-(g_m\cdot\na)u&=0\ \\  
~~~~~~~~~~~~~~~~~~~~~~~~~~~~~~~~~~~~~~~~~~~~~~~~~~~~~\na\cdot u =\na\cdot g_m&=0  \ 
\end{array}\right\}\text{ in }\R^3\times\R,\ m=1,2,3,\\
\lim_{|y_0|\to\infty}\int_{B_1(y_0)}\left(|u(y,s)-U_0(y,s)|^2+\underset{n=1}{\overset{3}\sum}|g_n(y,s)-G_{n,0}(y,s)|^2\right)dy=0\ \text{ for all }s\in\R,\\
u(\cdot,s)=u(\cdot,s+T),\ g_m(\cdot,s)=g_m(\cdot,s+T)\ \text{ in }\R^3\text{ for all }s\in\R,\ m=1,2,3,
\end{split}
\end{equation}
for given $T$-periodic divergence-free vector fields $U_0$ and $G_{m,0}$, $m=1,2,3$. 

Periodic weak solutions and suitable periodic weak solutions of \eqref{leray_vNSEd} are defined as follows.
\begin{defn}[Periodic weak solution of Leray system for the viscoelastic Navier-Stokes equations with damping] Let $U_0$ and $G_{m,0}$, $m=1,2,3$, satisfy \thref{assum_2.1}. A $4$-tuple of vector fields $(u,g_1,g_2,g_3)$ is a periodic weak solution to \eqref{leray_vNSEd} if for $m=1,2,3$ we have $\na\cdot u=\na\cdot g_m=0$, \[u-U_0,\,g_m-G_{m,0}\in L^\infty(0,T;L^2(\R^3))\cap L^2(0,T;H^1(\R^3)),\]
and 
\begin{equation}
\int_0^T\left[(u,\pd_s\varphi)-(\na u,\na \varphi)+\left(u+y\cdot\na u-u\cdot\na u+\underset{n=1}{\overset{3}\sum}(g_n\cdot\na)g_n,\varphi\right)\right]ds=0,
\end{equation}
\begin{equation}
\int_0^T\left[(g_m,\pd_s\varphi)-(\na g_m,\na \varphi)+(g_m+y\cdot\na g_m-u\cdot\na g_m+g_m\cdot\na u,\varphi)\right]ds=0,
\end{equation}
holds for all $\varphi\in\mathcal{D}_T$.
\end{defn}

\begin{defn}[Suitable periodic weak solution of Leray system for the viscoelastic Navier-Stokes equations with damping] Let $U_0$ and $G_{m,0}$, $m=1,2,3$, satisfy \thref{assum_2.1}.  
A $5$-tuple $(u,g_1,g_2, g_3,p)$ is a suitable periodic weak solution to \eqref{leray_vNSEd} if $u,g_1,g_2,g_3,p$ are periodic in $s$ with period $T$, $(u,g_1,g_2,g_3)$ is a periodic weak solution to \eqref{leray_vNSEd}, $p\in L^{3/2}_{\textup{loc}}(\R^4)$, $(u,g_1,g_2,g_3,p)$ solves \eqref{leray_vNSEd} in the sense of distributions, and the local energy inequality holds: 
\EQ{\label{lei_vNSEd_leray}
\int_{\R^4}\left(\frac{|u|^2+|{\bf G}|^2}2+|\na u|^2+|\na{\bf G}|^2\right)\psi \,dyds\le&\int_{\R^4}\frac{|u|^2+|{\bf G}|^2}2\left(\pd_s\psi+\De\psi\right)dyds\\&+\int_{\R^4}\left(\frac{|u|^2+|{\bf G}|^2}2(u-y)+pu\right)\cdot\na\psi\,dyds\\&-\underset{n=1}{\overset{3}\sum}\int_{\R^4}(u\cdot g_n)g_n\cdot\na\psi\,dyds,
} where ${\bf G}=(g_1,g_2,g_3)\in\R^{3\times3}$, for all nonnegative $\psi\in C^\infty_0(\R^4)$.
\end{defn}

The main result of this subsection can be stated as the following:
\begin{thm}[Existence of suitable periodic weak solutions to \eqref{leray_vNSEd}]\thlabel{thm_2.4_vNSEd}
Assume $U_0(y,s)$ and $G_{m,0}(y,s),\,m=1,2,3,$ all satisfy \thref{assum_2.1} with $q=10/3$. Then \eqref{leray_vNSEd} has a periodic suitable weak solution $(u,g_1,g_2,g_3,p)$ in $\R^4$ with period $T$.
\end{thm}
\begin{proof}
The proof follows from the same argument in that of \thref{thm_2.4_mhd}. Let $Z\in C^\infty(\R^3)$ with $0\le Z\le 1,\,Z(x)=1$ for $|x|>2$ and $Z(x)=0$ for $|x|<1$. 
Applying \thref{lem_2.5} with $\de=\frac18$, we are able to choose $R_0=R_0(U_0,G_{1,0},G_{2,0},G_{3,0})\ge1$ such that letting $\xi(y)=Z\left(\frac{y}{R_0}\right)$ and setting
\begin{equation}
W(y,s)=\xi(y)U_0(y,s)+w(y,s)
\end{equation}
and 
\begin{equation}
E_m(y,s)=\xi(y)G_{m,0}(y,s)+e_m(y,s),\ m=1,2,3,
\end{equation}
where 
 \begin{equation}
 w(y,s)=\int_{\R^3}\na_y\,\frac1{4\pi|y-z|}\,\na_z\xi(z)\cdot U_0(z,s)dz
 \end{equation}
 and
  \begin{equation}
 e_m(y,s)=\int_{\R^3}\na_y\,\frac1{4\pi|y-z|}\,\na_z\xi(z)\cdot G_{m,0}(z,s)dz,\ m=1,2,3,
 \end{equation}
 $W$ and $E_m,\,m=1,2,3,$ all satisfy the conclusion of \thref{lem_2.5}.

The Leray system \eqref{leray_vNSEd} can be written as 
\begin{equation}
\setlength\arraycolsep{1.5pt}\def\arraystretch{1.2}
\left\{\begin{array}{ll}
\mathcal{L}u+(u\cdot\na)u-\underset{n=1}{\overset{3}\sum}(g_n\cdot\na)g_n+\na p&=0\\
\mathcal{L}g_m+(u\cdot\na)g_m-(g_m\cdot\na)u&=0,\ m=1,2,3,\\
~~~~~~~~~~~~~~~~~~~~~~~~~~~~\na\cdot u=\na\cdot g_m&=0,\ m=1,2,3,
\end{array}\right.\end{equation}
where $\mathcal{L}$ is given in \eqref{diff_op_L}. We have to construct a solution of the form $u=U+W$ and $g_m=G_m+E_m,\,m=1,2,3$. It follows that $(U,G_1,G_2,G_3)$ satisfies the perturbed Leray system for the viscoelastic Navier-Stokes equations with damping
\begin{equation}\label{ptb_leray_vNSEd}
\setlength\arraycolsep{1.5pt}\def\arraystretch{1.2}
\left\{\begin{array}{ll}
&\mathcal{L}U+(W+U)\cdot\na U+U\cdot\na W\\
&~~~~~~~~~~~~~-\underset{n=1}{\overset{3}\sum}(E_n+G_n)\cdot\na G_n-\underset{n=1}{\overset{3}\sum}G_n\cdot\na E_n+\na p=-\mathcal{R}_3(W,E_1,E_2,E_3),\\
&\mathcal{L}G_m+(W+U)\cdot\na G_m+U\cdot\na E_m\\
&\,~~~~~~~~~~~~~~~~~~~~~~~~~~~~~~-(E_m+G_m)\cdot\na U-G_m\cdot\na W=-\mathcal{R}_4(W,E_m),\\
&~~~~~~~~~~~~~~~~~~~~~~~~~~~~~~~~~~~~~~~~~~~~~~~~~~~~\na\cdot U=\na\cdot G_m=0,
\end{array}\right.
\end{equation}
for $m=1,2,3,$ where 
\EQ{\label{eq_R3_R4}
\setlength\arraycolsep{1.5pt}\def\arraystretch{1.2}
\left\{\begin{array}{rl}
\mathcal{R}_3(W,E_1,E_2,E_3)&:=\mathcal{L}W+W\cdot\na W-\underset{n=1}{\overset{3}\sum}E_n\cdot\na E_n\\
\mathcal{R}_4(W,E_m)&:=\mathcal{L}E_m+W\cdot\na E_m-E_m\cdot\na W,\ m=1,2,3.
\end{array}\right.
}

We first solve the following mollified perturbed Leray system for the viscoelastic Navier-Stokes equations with damping for $(U^\ve,G_1^\ve,G_2^\ve,G_3^\ve,p^\ve)$ in $\R^3\times[0,T]$:
\begin{equation}\label{mdf_ptb_leray_vNSEd}
\setlength\arraycolsep{1.5pt}\def\arraystretch{1.2}
\left\{\begin{array}{ll}
&\mathcal{L}U^\ve+(W+(\eta_\ve*U^\ve))\cdot\na U^\ve+U^\ve\cdot\na W\\
&~~~~~~-\underset{n=1}{\overset{3}\sum}(E_n+(\eta_\ve*G_n^\ve))\cdot\na G_n^\ve-\underset{n=1}{\overset{3}\sum}G_n^\ve\cdot\na E_n+\na p=-\mathcal{R}_3(W,E_1,E_2,E_3),\\
&\mathcal{L}G_m^\ve+(W+(\eta_\ve*U^\ve))\cdot\na G_m^\ve+U^\ve\cdot\na E_m\\
&\,~~~~~~~~~~~~~~~~~~~~~-(E_m+(\eta_\ve*G_m^\ve))\cdot\na U^\ve-G_m^\ve\cdot\na W=-\mathcal{R}_4(W,E_m),\\
&~~~~~~~~~~~~~~~~~~~~~~~~~~~~~~~~~~~~~~~~~~~~~~~~~~~\na\cdot U^\ve=\na\cdot G_m^\ve=0,
\end{array}\right.
\end{equation}
for $m=1,2,3$, where $\eta_\ve(y)=\ve^{-3}\eta(y/\ve)$ for some fixed function $\eta\in C^\infty_0(\R^3)$ satisfying $\int_{\R^3}\eta dy=1$. It has the following weak formulation:
\begin{equation}
\setlength\arraycolsep{1.5pt}\def\arraystretch{1.2}
\left\{\begin{array}{lll}
\frac{d}{ds}(U^\ve,f)&=&-(\na U^\ve,\na f)+(U^\ve+y\cdot\na U^\ve,f)\\
&&-\left((\eta_\ve*U^\ve)\cdot\na U^\ve-\underset{n=1}{\overset{3}\sum}(\eta_\ve*G_n^\ve)\cdot\na G_n^\ve,f\right)\\
&&-\left(W\cdot\na U^\ve+U^\ve\cdot\na W-\underset{n=1}{\overset{3}\sum}E_n\cdot\na G_n^\ve-\underset{n=1}{\overset{3}\sum}G_n^\ve\cdot\na E_n,f\right)\\
&&-\left<\mathcal{R}_3(W,E_1,E_2,E_3),f\right>\\
\frac{d}{ds}(G_m^\ve,f)&=&-(\na G_m^\ve,\na f)+(G_m^\ve+y\cdot\na G_m^\ve,f)\\
&&-\left((\eta_\ve*U^\ve)\cdot\na G_m^\ve-(\eta_\ve*G_m^\ve)\cdot\na U^\ve,f\right)\\
&&-(W\cdot\na G_m^\ve+U^\ve\cdot\na E_m-E_m\cdot\na U^\ve-G_m^\ve\cdot\na W,f)\\
&&-\left<\mathcal{R}_4(W,E_m),f\right>,\ m=1,2,3,
\end{array}\right.
\end{equation}
for all $f\in\mathcal{V}$ and a.e. $s\in(0,T)$.

~\\
{\bf Step 1: Construction of a solution to the mollified perturbed Leray system}

We use the Galerkin method to construct a solution of \eqref{mdf_ptb_leray_vNSEd}. Let $\{h_k\}_{k\in\NN}\subset\mathcal{V}$ be an orthonormal basis of $H$. For a fixed $k\in\NN$, we look for an approximation solution of the form $U^\ve_k(y,s)=\sum_{i=1}^k\mu^\ve_{ki}(s)h_i(y),\,(G^\ve_m)_k(y,s)=\sum_{i=1}^k(\ga^\ve_m)_{ki}(s)h_i(y),\,m=1,2,3$. First, we prove the existence and derive an a priori bound for $T$-periodic solutions $\mu^\ve_k=(\mu^\ve_{k1},\cdots,\mu^\ve_{kk}),\,(\ga^\ve_m)_k=((\ga^\ve_m)_{k1},\cdots,(\ga^\ve_m)_{kk}),\,m=1,2,3,$ to the system of ODEs
\begin{equation}\label{galerkin_ode_vNSEd}
\setlength\arraycolsep{1.5pt}\def\arraystretch{1.2}
\left\{\begin{array}{ll}
\frac{d}{ds}\mu^\ve_{kj}&=\underset{i=1}{\overset{k}\sum}\mathscr{A}_{ij}\mu^\ve_{ki}+\underset{i=1}{\overset{k}\sum}\,\underset{n=1}{\overset{3}\sum}\tilde{\mathscr{B}}_{ijn}(\ga^\ve_n)_{ki}+\underset{i,l=1}{\overset{k}\sum}\mathscr{C}^\ve_{ilj}\mu^\ve_{ki}\mu^\ve_{kl}-\underset{i,l=1}{\overset{k}\sum}\mathscr{C}^\ve_{ilj}\underset{n=1}{\overset{3}\sum}(\ga^\ve_n)_{ki}(\ga^\ve_n)_{kl}+\tilde{\mathscr{D}}_j\\
\frac{d}{ds}(\ga^\ve_m)_{kj}&=\underset{i=1}{\overset{k}\sum}\tilde{\mathscr{E}}_{ijm}\mu^\ve_{ki}+\underset{i=1}{\overset{k}\sum}\mathscr{F}_{ij}(\ga^\ve_m)_{ki}+\underset{i,l=1}{\overset{k}\sum}\mathscr{G}^\ve_{ilj}\mu^\ve_{ki}(\ga^\ve_m)_{kl}+\tilde{\mathscr{H}}_{jm},
\end{array}\right.
\end{equation}
for $j=1,\cdots,k$, where $\mathscr{A}_{ij},\,\mathscr{C}^\ve_{ilj},\,\mathscr{F}_{ij}$ and $\mathscr{G}^\ve_{ilj}$ are the same as those in \eqref{galerkin_ode_coeff_mhd}, and 
\EQ{
\tilde{\mathscr{B}}_{ijn}&=(h_i\cdot\na E_n,h_j)+(E_n\cdot\na h_i,h_j),\\
\tilde{\mathscr{D}}_j&=-\left<\mathcal{R}_3(W,E_1,E_2,E_3),h_j\right>,\\
\tilde{\mathscr{E}}_{ijm}&=-(h_i\cdot\na E_m,h_j)+(E_m\cdot\na h_i,h_j),\\
\tilde{\mathscr{H}}_{jm}&=-\left<\mathcal{R}_4(W,E_m),h_j\right>.
}
Fix any $k\in\NN$. For any $U^0,G_m^0\in\textup{span}(h_1,\cdots,h_k),\,m=1,2,3$, there exist $\mu^\ve_{kj},(\ga^\ve_m)_{kj}\in H^1(0,\tilde{T}),\,j=1,\cdots,k$, that uniquely solve \eqref{galerkin_ode_vNSEd} with initial data $\mu^\ve_{kj}(0)=(U^0,h_j),\,(\ga^\ve_m)_{kj}(0)=(G_m^0,h_j)$, $j=1,\cdots,k$, for some $0<\tilde{T}\le T$. 

We prove that $\tilde{T}=T$. Indeed, multiplying the $j$-th equation of $\eqref{galerkin_ode_vNSEd}_1$ by $\mu^\ve_{kj}$, multiplying the $j$-th equation of $\eqref{galerkin_ode_vNSEd}_2$ by $(\ga^\ve_m)_{kj}$, and summing over all $j=1,\cdots,k$ and $m=1,2,3$, that yields
\EQ{\label{eq_2.27_vNSEd}
&~~~~\frac12\,\frac{d}{ds}\left(\|U^\ve_k\|_{L^2}^2+\underset{n=1}{\overset{3}\sum}\|(G^\ve_n)_k\|_{L^2}^2\right)+\frac12\left(\|U^\ve_k\|_{L^2}^2+\underset{n=1}{\overset{3}\sum}\|(G^\ve_n)_k\|_{L^2}^2\right)\\
&~~~~~~~~~~~~~~~~~~~~~~~~~~~~~~~~~~~~~~~~~~~~~~~~~~~~~~~~~~~~~~+\left(\|\na U^\ve_k\|_{L^2}^2+\underset{n=1}{\overset{3}\sum}\|\na  (G^\ve_n)_k\|_{L^2}^2\right)\\
&=-\left(U^\ve_k\cdot\na W-\underset{n=1}{\overset{3}\sum}(E_n\cdot\na (G^\ve_n)_k+(G^\ve_n)_k\cdot\na E_n),U^\ve_k\right)\\
&~~~~~~~~~~~~~~~~~~~~~~~~~~~~~~~~~~~~~~-\underset{n=1}{\overset{3}\sum}((U^\ve_k\cdot\na E_n-E_n\cdot\na U^\ve_k)-(G^\ve_n)_k\cdot\na W,(G^\ve_n)_k)\\
&~~~~~~~~~~~~~~~~~~~~~~~~~~~~~~~~~~~~~~-\left<\mathcal{R}_3(W,E_1,E_2,E_3),U^\ve_k\right>-\underset{n=1}{\overset{3}\sum}\left<\mathcal{R}_4(W,E_n),(G^\ve_n)_k\right>,
}
thanks to the vanishing of $((\eta_\ve*U^\ve)\cdot\na U^\ve,U^\ve),\,(W\cdot\na U^\ve,U^\ve),\,((\eta_\ve*U^\ve)\cdot\na G^\ve_m,G^\ve_m)$ and $(W\cdot\na G^\ve_m,G^\ve_m)$, and the cancellation of $\sum_{n=1}^3((\eta_\ve*G^\ve_n)\cdot\na G^\ve_n,U^\ve)$ and $((\eta_\ve*G^\ve_m)\cdot\na U^\ve,G^\ve_m)$. Using \thref{lem_2.5} with $\de=\frac18$, we get 
\EQ{\label{etm_2.28_vNSEd}
&~~~\left|-\left(U^\ve_k\cdot\na W-\underset{n=1}{\overset{3}\sum}(E_n\cdot\na (G^\ve_n)_k+(G^\ve_n)_k\cdot\na E_n),U^\ve_k\right)\right.\\
&~~~~\left.-\underset{n=1}{\overset{3}\sum}((U^\ve_k\cdot\na E_n-E_n\cdot\na U^\ve_k)-(G^\ve_n)_k\cdot\na W,(G^\ve_n)_k)\right|\\
&\le\frac1{16}\left(7\,\|U^\ve_k\|_{H^1}^2+3\,\underset{n=1}{\overset{3}\sum}\|(G^\ve_n)_k\|_{H^1}^2\right),
}
and
\EQ{\label{etm_2.29_vNSEd}
&~~~\left|-\left<\mathcal{R}_3(W,E_1,E_2,E_3),U^\ve_k\right>-\underset{n=1}{\overset{3}\sum}\left<\mathcal{R}_4(W,E_n),(G^\ve_n)_k\right>\right|\\
&\le C_2+\frac1{128}\left(5\,\|U^\ve_k\|_{H^1}^2+3\,\underset{n=1}{\overset{3}\sum}\|(G^\ve_n)_k\|_{H^1}^2\right),
}
where $C_2=32\left(\|\mathcal{L}W\|_{H^{-1}}^2+\sum_{n=1}^3\|\mathcal{L}E_n\|_{H^{-1}}^2+(\|W\|_{L^4}^2+\sum_{n=1}^3\|E_n\|_{L^4}^2)^2\right)$ is independent of $s,T,k$ and $\varepsilon$.

Using the estimates \eqref{etm_2.28_vNSEd} and \eqref{etm_2.29_vNSEd}, we obtain from \eqref{eq_2.27_vNSEd} the differential inequality
\EQ{\label{eq_2.30_vNSEd}
&\frac{d}{ds}\left(\|U^\ve_k\|_{L^2}^2+\underset{n=1}{\overset{3}\sum}\|(G^\ve_n)_k\|_{L^2}^2\right)+\frac1{64}\left(\|U^\ve_k\|_{L^2}^2+\underset{n=1}{\overset{3}\sum}\|(G^\ve_n)_k\|_{L^2}^2\right)\\
&~~~~~~~~~~~~~~~~~~~~~~~~~~~~~~~~~~~~~~~~~~~~~~~~~~~+\frac1{64}\left(\|\na U^\ve_k\|_{L^2}^2+\underset{n=1}{\overset{3}\sum}\|\na(G^\ve_n)_k\|_{L^2}^2\right)\le C_2.
}
The Gronwall inequality implies that 
\EQ{\label{eq_2.31_vNSEd}
e^{s/{64}}\left(\|U^\ve_k\|_{L^2}^2+\underset{n=1}{\overset{3}\sum}\|(G^\ve_n)_k\|_{L^2}^2\right)\le&~\left(\|U^0\|_{L^2}^2+\underset{n=1}{\overset{3}\sum}\|G_n^0\|_{L^2}^2\right)+\int_0^{\tilde{T}}e^{\tau/{64}}C_2d\tau\\
\le&~\left(\|U^0\|_{L^2}^2+\underset{n=1}{\overset{3}\sum}\|G_n^0\|_{L^2}^2\right)+e^{T/{64}}C_2T
}
for all $s\in[0,\tilde{T}]$. Since the right-hand side is finite, $\tilde{T}$ is not a blow-up time and we conclude that $\tilde{T}=T$.

Choosing $\rho=\frac{C_2T}{1-e^{-T/{64}}}>0$ (independent of $k$), \eqref{eq_2.31_vNSEd} implies that \[\left(\|U^\ve_k\|_{L^2}^2+\sum_{n=1}^3\|(G^\ve_n)_k\|_{L^2}^2\right)^{1/2}\le\rho\] if $\left(\|U^0\|_{L^2}^2+\underset{n=1}{\overset{3}\sum}\|G_n^0\|_{L^2}^2\right)^{1/2}\le\rho$. Define $\mathcal{T}:B_\rho^{4k}\to B_\rho^{4k}$ by \[\mathcal{T}(\mu^\ve_k(0),(\ga^\ve_1)_k(0),(\ga^\ve_2)_k(0),(\ga^\ve_3)_k(0))=(\mu^\ve_k(T),(\ga^\ve_1)_k(T),(\ga^\ve_2)_k(T),(\ga^\ve_3)_k(T)),\] where $B_\rho^{4k}$ is the closed ball in $\R^{4k}$ of radius $\rho$ and centered at the origin. According to the continuous dependence on initial conditions of the solution of ODEs, the map $\mathcal{T}$ is continuous. Thus, we can find a fixed point of $\mathcal{T}$ by the Brouwer fixed point theorem. That is, there exist $(\mu^\ve_k(0),(\ga^\ve_1)_k(0),(\ga^\ve_2)_k(0),(\ga^\ve_3)_k(0))\in B_\rho^{4k}$ such that $(\mu^\ve_k(0),(\ga^\ve_1)_k(0),(\ga^\ve_2)_k(0),(\ga^\ve_3)_k(0))=(\mu^\ve_k(T),(\ga^\ve_1)_k(T),(\ga^\ve_2)_k(T),(\ga^\ve_3)_k(T))$. Let $U^0=\sum_{i=1}^k\mu_{ki}(0)h_i$ and $G_m^0=\sum_{i=1}^k(\ga_m)_{ki}(0)h_i$. Then $U^0,G_m^0\in\textup{span}(h_1,\cdots,h_k)$ and $U^0=U^\ve_k(T),G_m^0=(G^\ve_m)_k(T)$.

We have $\left(\|U^\ve_k(s)\|_{L^2}^2+\underset{n=1}{\overset{3}\sum}\|(G^\ve_n)_k(s)\|_{L^2}^2\right)^\frac12\le\rho$ for all $s\in[0,T]$ by the choice of $U^0$ and $G_m^0$. Hence
\begin{equation}
\left(\|U^\ve_k\|_{L^\infty(0,T;L^2(\R^3))}^2+\sum_{n=1}^3\|(G^\ve_n)_k\|_{L^\infty(0,T;L^2(\R^3))}^2\right)^\frac12\le\rho.
\end{equation}
Moreover, by integrating \eqref{eq_2.30_vNSEd} in $s\in[0,T]$ and using $U^\ve_k(0)=U^\ve_k(T),\,(G^\ve_m)_k(0)=(G^\ve_m)_k(T)$, we get 
\begin{equation}
\frac1{64}\int_0^T\left(\|U^\ve_k(s)\|_{H^1}^2+\sum_{n=1}^3\|(G^\ve_n)_k(s)\|_{H^1}^2\right)ds\le C_2T.
\end{equation}
Therefore, 
\EQ{\label{eq_2.26_vNSEd}
&\|U^\ve_k\|_{L^\infty(0,T;L^2(\R^3))}+\sum_{n=1}^3\|(G^\ve_n)_k\|_{L^\infty(0,T;L^2(\R^3))}\\&~~~~~~~~~~~~~~~~~~~~~~~~~~~~~~~~~~~+\|U^\ve_k\|_{L^2(0,T;H^1(\R^3))}+\sum_{n=1}^3+\|(G^\ve_n)_k\|_{L^2(0,T;H^1(\R^3))}\le C,
}
where $C=\sqrt{8(\rho^2+64C_2T)}$ is independent of both $\ve$ and $k$.

Since the sequences $\{U^\ve_k\}_{k\in\NN}$ and $\{(G^\ve_m)_k\}_{k\in\NN}$ are uniformly bounded, a standard limiting process shows that, for all $\ve>0$, we have, up to some subsequences, that 
\EQ{
&U^\ve_k\rightharpoonup U^\ve,\ (G^\ve_m)_k\rightharpoonup G^\ve_m\ \text{ weakly in $L^2(0,T;X)$},\\
&U^\ve_k\to U^\ve,\ (G^\ve_m)_k\to G^\ve_m\ \text{ strongly in $L^2(0,T;L^2(K))$ for all compact sets $K\subset\R^3$},\\
&U^\ve_k(s)\rightharpoonup U^\ve(s),\ (G^\ve_m)_k(s)\rightharpoonup G^\ve_m(s)\ \text{ weakly in $L^2$ for all $s\in[0,T]$}
}
as $k\to\infty$, for some $U^\ve,(G^\ve_m)\in L^2(0,T;H^1_0(\R^3)),\,m=1,2,3,$ (all have $\ve$-independent $L^\infty L^2$ and $L^2H^1$ bounds). The weak convergence ensures that $U^\ve(0)=U^\ve(T)$ and $(G^\ve_m)(0)=G^\ve_m(T)$. Furthermore, the $4$-tuple $(U^\ve,G^\ve_1,G^\ve_2,G^\ve_3)$ is a periodic weak solution of the mollified perturbed Leray system \eqref{mdf_ptb_leray_vNSEd}.

~\\
{\bf Step 2: A priori estimate of the pressure in the mollified perturbed Leray system}

By taking the divergence of $\eqref{mdf_ptb_leray_vNSEd}_1$, we obtain
\EQ{\label{eq_2.33_vNSEd}
-\De p^\ve=&\sum_{i,j=1}^k\pd_i\pd_j\,\Bigg[\,(\eta_\ve*U^\ve_i)U^\ve_j+W_iU^\ve_j+U^\ve_i W_j+W_iW_j\\
&~~~~~~~~~~~~~~~\left.-\sum_{n=1}^3\left((\eta_\ve*(G^\ve_n)_i)(G^\ve_n)_j+(E_n)_i(G^\ve_n)_j+(G^\ve_n)_i(E_n)_j+(E_n)_i(E_n)_j\right)\right].
}
Let 
\EQ{
\tilde{p}^\ve=&\sum_{i,j=1}^kR_iR_j\,\Bigg[\,(\eta_\ve*U^\ve_i)U^\ve_j+W_iU^\ve_j+U^\ve_i W_j+W_iW_j\\
&~~~~~~~~~~~~~~~~\left.-\sum_{n=1}^3\left((\eta_\ve*(G^\ve_n)_i)(G^\ve_n)_j+(E_n)_i(G^\ve_n)_j+(G^\ve_n)_i(E_n)_j+(E_n)_i(E_n)_j\right)\right],
}
where $R_i$ denote the Riesz transforms. Note that $\tilde{p}^\ve$ also satisfies \eqref{eq_2.33_vNSEd}. We will prove $\na(p^\ve-\tilde{p}^\ve)=0$ so that $p^\ve=\tilde{p}^\ve$ up to an additive constant by proving . 

Let $V^\ve(x,t)=(2t)^{-1/2}U^\ve(y,s),\,\pi^\ve(x,t)=(2t)^{-1}p^\ve(y,s)$ and $\tilde{\mathcal{F}}^\ve(x,t)=(\tilde{\mathcal{F}}_1+\tilde{\mathcal{F}}^\ve_2)(x,t)$ where 
\begin{equation}
\tilde{\mathcal{F}}_1(x,t):=-\frac1{(2t)^{3/2}}(\mathcal{L}W)(y,s),
\end{equation}
\EQ{
\tilde{\mathcal{F}}^\ve_2(x,t):=-\frac1{(2t)^{3/2}}&\, \Bigg[\, W\cdot\na U^\ve+U^\ve\cdot\na W+(\eta_\ve*U^\ve)\cdot\na U^\ve+W\cdot\na W\\
&~~\left.-\sum_{n=1}^3\left(E_n\cdot\na G^\ve_n-G^\ve_n\cdot\na E_n-(\eta_\ve*G^\ve_n)\cdot\na G^\ve_n-E_n\cdot\na E_n\right)\right](y,s),
}
and $y=x/\sqrt{2t}$ and $s=\log(\sqrt{2t})$. Hence, $\tilde{\mathcal{F}}^\ve\in L^\infty(1,\la^2;H^{-1}(\R^3))$, $(V^\ve,\pi)$ solves the non-stationary Stokes system on $\R^3\times[1,\la^2]$ with force $\tilde{\mathcal{F}}^\ve$ by $\eqref{mdf_ptb_leray_vNSEd}_1$, and $V^\ve$ is in the energy class. In view of the uniqueness of the solution to the forced, non-stationary Stokes system on $\R^3\times[1,\la^2]$, we can conclude that $\na\pi^\ve=\na\tilde{\pi}^\ve$ where $\tilde{\pi}^\ve=(2t)^{-1}\tilde{p}^\ve$. Therefore $\na(p^\ve-\tilde{p}^\ve)=0$.

At this point, we may replace $p^\ve$ by $\tilde{p}^\ve$. As before, the Calder\'on-Zygmund theory gives 
\EQN{
\|p^\ve(s)\|_{L^{5/3}}\le&~C\, \Bigg\|\,\Big[ (\eta_\ve*U^\ve_i)U^\ve_j+W_iU^\ve_j+U^\ve_i W_j+W_iW_j\\
&~~~~~-\left.\sum_{n=1}^3\left((\eta_\ve*(G^\ve_n)_i)(G^\ve_n)_j(E_n)_i(G^\ve_n)_j+(G^\ve_n)_i(E_n)_j+(E_n)_i(E_n)_j\right)\Big](s)\right\|_{L^{5/3}}
\\\le&~C\left(\|U^\ve(s)\|_{L^{10/3}}^2+\sum_{n=1}^3\|G^\ve_n(s)\|_{L^{10/3}}^2+\|W(s)\|_{L^{10/3}}^2+\sum_{n=1}^3\|E_n(s)\|_{L^{10/3}}^2\right).
}
So we get the following a priori bound for $p^\ve$:
\EQ{\label{eq_2.37_vNSEd}
\|p^\ve\|_{L^{5/3}(\R^3\times[0,T])}\le&~C\left(\|U^\ve\|_{L^{10/3}(\R^3\times[0,T])}^2+\sum_{n=1}^3\|G^\ve_n\|_{L^{10/3}(\R^3\times[0,T])}^2\right.\\
&~~~~~~+\left.\|W\|_{L^{10/3}(\R^3\times[0,T])}^2+\sum_{n=1}^3\|E_n\|_{L^{10/3}(\R^3\times[0,T])}^2\right).
}
Since the sequences $\{U^\ve\}_{\ve>0}$ and $\{G^\ve_m\}_{\ve>0},\,m=1,2,3,$ are both bounded in $L^\infty L^2$ and $L^2H^1$ norms,
\EQ{\label{bound_U^ve_vNSEd}
\|U^\ve\|_{L^{10/3}(\R^3\times[0,T])}=\left\|\|U^\ve\|_{L_y^{10/3}}\right\|_{L_s^{10/3}}\le&~\left\|\|U^\ve\|_{L_y^2}^{\frac25}\|U^\ve\|_{L_y^6}^{\frac35}\right\|_{L_s^{10/3}}\\
\le&~\|U^\ve\|_{L^\infty(0,T;L^2(\R^3))}^{\frac25}\left\|\|U^\ve\|_{L_y^6}^{\frac35}\right\|_{L_s^{10/3}}\\
\le&~\|U^\ve\|_{L^\infty(0,T;L^2(\R^3))}^{\frac25}\|U^\ve\|_{L^2(0,T;L^6(\R^3))}^{\frac35}\\
\lesssim&~\|U^\ve\|_{L^\infty(0,T;L^2(\R^3))}^{\frac25}\|U^\ve\|_{L^2(0,T;H^1(\R^3))}^{\frac35}\\
\le&~C,
}
where $C$ is some constant independent of $\ve$. Similarly, we have 
\EQ{\label{bound_A^ve_vNSEd}
\|G^\ve_m\|_{L^{10/3}(\R^3\times[0,T])}\le C,\ m=1,2,3.
}
Moreover, since we are applying \thref{lem_2.5} with $q=\frac{10}3$ and $\de=\frac18$, $\|W\|_{L^\infty(0,T;L^{10/3}(\R^3))}\le\frac18$ and $\|E_m\|_{L^\infty(0,T;L^{10/3}(\R^3))}\le\frac18$. Thus, we have the estimates
\EQ{\label{bound_WE}
\|W\|_{L^{10/3}(\R^3\times[0,T])}\le \frac18\,T^{10/3}\ \ \ \text{ and }\ \ \ \|E_m\|_{L^{10/3}(\R^3\times[0,T])}\le \frac18\,T^{10/3},\ m=1,2,3.
}

Using the bounds \eqref{bound_U^ve_vNSEd}-\eqref{bound_WE}, \eqref{eq_2.37_vNSEd} implies that $\{p^\ve\}_{\ve>0}$ is a bounded sequence in $L^{5/3}(\R^3\times[0,T])$.

~\\
{\bf Step 3: Convergence to a suitable periodic weak solution to \eqref{leray_vNSEd}}

On one hand, since the sequences $\{U^\ve\}_{\ve>0}$ and $\{G^\ve_m\}_{\ve>0},\,m=1,2,3,$ are all bounded in $L^\infty(0,T;L^2(\R^3))$- and $L^2(0,T;H^1(\R^3))$- norms, there exist $U,G_m\in L^\infty(0,T;L^2(\R^3))\cap L^2(0,T;H^1_0(\R^3)),\,m=1,2,3,$ and sequences $\{U^{\ve_k}\}_{k\in\NN},\{G^{\ve_k}_m\}_{k\in\NN},\,m=1,2,3,$ such that for $m=1,2,3,$
\EQ{\label{U_G_conv}
&U^{\ve_k}\rightharpoonup U,\ G^{\ve_k}_m\rightharpoonup G_m\ \ \ \text{ weakly in $L^2(0,T;X)$},\\
&U^{\ve_k}\to U,\ G^{\ve_k}_m\to G_m\ \ \ \text{ strongly in $L^2(0,T;L^2(K))$ for all compact sets $K\subset\R^3$},\\
&U^{\ve_k}(s)\rightharpoonup U(s),\ G^{\ve_k}_m(s)\rightharpoonup G_m(s)\ \ \ \text{ weakly in $L^2$ for all $s\in[0,T]$},
}
as $\ve_k\to0$.

On the other hand, since $\{p^{\ve_k}\}_{k\in\NN}$ is a bounded sequence in $L^{5/3}(\R^3\times[0,T])$, we have that 
\EQ{
p^{\ve_k}\rightharpoonup p\ \text{ weakly in $L^{5/3}(\R^3\times[0,T])$,}
}
for some $p\in L^{5/3}(\R^3\times[0,T])$. Let $u=U+W$ and $g_m=G_m+E_m,\,m=1,2,3$. The above convergences are strong enough to guarantee that the $5$-tuple $(u,g_1,g_2,g_3,p)$ solves \eqref{leray_vNSEd} in the sense of distributions.

What is left is to show that $(u,g_1,g_2,g_3,p)$ satisfies the local energy inequality \eqref{lei_vNSEd_leray}. Note that $(u^{\ve_k},g^{\ve_k}_1,g^{\ve_k}_2,g^{\ve_k}_3,p^{\ve_k})$, where $u^{\ve_k}=U^{\ve_k}+W$ and $g^{\ve_k}_m=G^{\ve_k}_m+E_m,\,m=1,2,3$, satisfies
\begin{equation}\label{eq_u^ve_g^ve}
\setlength\arraycolsep{1.5pt}\def\arraystretch{1.2}
\left\{\begin{array}{ll}
&\mathcal{L}u^{\ve_k}+W\cdot\na u^{\ve_k}+(\eta_{\ve_k}*U^{\ve_k})\cdot\na U^{\ve_k}+U^{\ve_k}\cdot\na W\\
&~~~~~~~~~~~~~-\underset{n=1}{\overset{3}\sum}E_n\cdot\na g^{\ve_k}_n-\underset{n=1}{\overset{3}\sum}(\eta_{\ve_k}*G^{\ve_k}_n)\cdot\na G^{\ve_k}_n-\underset{n=1}{\overset{3}\sum}G^{\ve_k}_n\cdot\na E_n+\na p^{\ve_k}=0\\
&\mathcal{L}g^{\ve_k}_m+W\cdot\na g^{\ve_k}_m+(\eta_{\ve_k}*U^{\ve_k})\cdot\na G^{\ve_k}_m+U^{\ve_k}\cdot\na E_m\\
&~~~~~~~~~~~~~~~~~~~~~~~~~~~~~~~~~~~~-E_m\cdot\na u^{\ve_k}-(\eta_{\ve_k}*G^{\ve_k}_m)\cdot\na U^{\ve_k}-G^{\ve_k}_m\cdot\na W=0.
\end{array}\right.
\end{equation}
Testing $\eqref{eq_u^ve_g^ve}_1$ and $\eqref{eq_u^ve_g^ve}_2$ for $m=1,2,3,$ with $u^{\ve_k}\psi$ and $g^{\ve_k}_m\psi,\,m=1,2,3$, respectively, where $0\le\psi\in C^\infty_0(\R^4)$ and adding them together, we get
\EQ{\label{A_51_vNSEd}
&~~~~\int_{\R^4}\left[\frac12\left(|u^{\ve_k}|^2+\underset{n=1}{\overset{3}\sum}\,|g^{\ve_k}_n|^2\right)+|\na u^{\ve_k}|^2+\underset{n=1}{\overset{3}\sum}\,|\na g^{\ve_k}_n|^2\right]\psi \,dyds\\
&=\int_{\R^4}\frac12\left(|u^{\ve_k}|^2+\underset{n=1}{\overset{3}\sum}\,|g^{\ve_k}_n|^2\right)\left(\pd_s\psi+\De\psi\right)dyds\\
&~~~+\int_{\R^4}\frac12\left(|u^{\ve_k}|^2+\underset{n=1}{\overset{3}\sum}\,|g^{\ve_k}_n|^2\right)(W-y)\cdot\na\psi\,dyds\\
&~~~+\int_{\R^4}\left[\frac12\left(|U^{\ve_k}|^2+2(U^{\ve_k}\cdot W)+\underset{n=1}{\overset{3}\sum}\,\left(|G_n^{\ve_k}|^2+2(G_n^{\ve_k}\cdot E_n)\right)\right)\left(\eta_{\ve_k}*U^{\ve_k}\right)\right.\\
&~~~~~~~~~~~~~~~~~~~~~~~~~~~~~~~~~~~~~~~~~~~~~~~~~~~~~~\left.+\frac12\left(|W|^2+\underset{n=1}{\overset{3}\sum}\,|E_n|^2\right)U^{\ve_k}\right]\cdot\na\psi\,dyds\\
&~~~+\int_{\R^4}p^{\ve_k} u^{\ve_k}\cdot\na\psi\,dyds\\
&~~~-\underset{n=1}{\overset{3}\sum}\,\int_{\R^4}\left[(u^{\ve_k}\cdot g^{\ve_k}_n)E_n+(U^{\ve_k}\cdot G^{\ve_k}_n)(\eta_{\ve_k}*G^{\ve_k}_n)+(U^{\ve_k}\cdot E_n)G^{\ve_k}_n\right.\\
&~~~~~~~~~~~~~~~~~~~~~~~~~~~~~~~~~~~~~~~~~~~~~~~~~~~~\left.+(W\cdot G^{\ve_k}_n)G^{\ve_k}_n+(W\cdot E_n)G^{\ve_k}_n\right]\cdot\na\psi\,dyds\\
&~~~+\int_{\R^4}\left((\eta_{\ve_k}*U^{\ve_k})-U^{\ve_k}\right)\cdot\left(\na W\cdot U^{\ve_k}+\underset{n=1}{\overset{3}\sum}\,\na E_n\cdot G^{\ve_k}_n\right)\psi\,dyds\\
&~~~+\underset{n=1}{\overset{3}\sum}\int_{\R^4}((\eta_{\ve_k}*G^{\ve_k}_n)-G^{\ve_k}_n)\cdot(\na U^{\ve_k}\cdot E_n+\na G^{\ve_k}_n\cdot W)\psi\,dyds.
}
Let $\mathcal{K}$ be a compact subset of $\R^4$. Using the same argument deriving \eqref{eta_U_conv_mhd} and \eqref{eta_A_conv_mhd}, we have 
\EQ{\label{eta_U_conv_vNSEd}
\|(\eta_{\ve_k}*U^{\ve_k})-U\|_{L^2(\mathcal{K})}\to0\ \text{ as }\ \ve_k\to0\ \ \ \text{ for all compact }\mathcal{K}\subset\R^4,
} 
and, for $m=1,2,3$,
\EQ{\label{eta_G_conv_vNSEd}
\|(\eta_{\ve_k}*G^{\ve_k}_m)-G_m\|_{L^2(\mathcal{K})}\to0\ \text{ as }\ \ve_k\to0\ \ \ \text{ for all compact }\mathcal{K}\subset\R^4.
}
In addition, the sequence $\{u^\ve\}_{\ve>0}$ is bounded in $L^{10/3}(\R^3\times[0,T])$ since it is bounded in $L^2(0,T;L^6(\R^3))$ and $L^\infty(0,T;L^2(\R^3))$. As before, we use the fact in the Appendix of \cite{MR673830} to show that
\EQ{\label{u_strong_5/2_vNSEd}
u^{\ve_k}\to u\ \ \ \text{ strongly in }L^{5/2}(\mathcal{K})\ \text{ as }\ve_k\to0.
}

Combining \eqref{eta_U_conv_vNSEd}-\eqref{u_strong_5/2_vNSEd} and the convergences in \eqref{U_G_conv} with the facts that $W,\,E_m$ are locally differentiable and that $\psi$ is compactly supported, each term on the right hand side of \eqref{A_51_vNSEd} converges to the corresponding term involving $u,\,U,\,g_m,\,G_m$ and $p$. Passing limit as $\ve_k\to0$, we get the desired local energy inequality \eqref{lei_vNSEd_leray} since $\int\na |u^{\ve_k}|^2dyds$ and $\int|\na g^{\ve_k}_m|^2dyds$ are lower-semicontinuous as $\ve_k\to0$. This proves \thref{thm_2.4_vNSEd}.
\end{proof}

\section{Discretely Self-Similar Solutions}
In this section, we prove \thref{thm_1.2_mhd} and \thref{thm_1.2_vNSEd}.
\subsection{Discretely self-similar solutions to the MHD equations}\label{sect_3.1}
\begin{proof}[Proof of \thref{thm_1.2_mhd}]
Let $U_0(y,s)=\sqrt{2t}(e^{t\De}v_0)(x)$ and $A_0=\sqrt{2t}(e^{t\De}b_0)(x)$. By \thref{lem_3.4}, $U_0$ and $A_0$ both satisfy \thref{assum_2.1} with $T=\log\la$ and $q=10/3$. Let $(u,a,p)$ be the $T$-periodic weak solution derived in \thref{thm_2.4_mhd}. Let $v(x,t)=u(y,s)/\sqrt{2t},\,b(x,t)=a(y,s)/\sqrt{2t}$ and $\pi(x,t)=p(y,s)/2t$ where $x,t,y,s$ satisfy \eqref{xtys}. Then $(v,b,\pi)$ is a distributional solution to \eqref{MHD}.

Note that $u-U_0$ is periodic in $s$ with period $T=\log(\la)$. So
\EQN{
\|v-e^{t\De} v_0\|_{L_t^\infty(1,\la^2;L_x^2(\R^3))}^2\le&~\la\|u-U_0\|_{L_s^\infty\left(\frac12\log2,\frac12\log2+\log(\la);L_y^2(\R^3)\right)}^2\\
\le&~\la\|u-U_0\|_{L_s^\infty\left(0,T;L_y^2(\R^3)\right)}^2.
}
Similarly, we have
\[\|v-e^{t\De} v_0\|_{L_t^2(1,\la^2;L_x^2(\R^3))}^2\le\la^3\|u-U_0\|_{L_s^2\left(0,T;L_y^2(\R^3)\right)}^2,\]
and 
\[\|\na_x\left(v-e^{t\De} v_0\right)\|_{L_t^2(1,\la^2;L_x^2(\R^3))}^2\le\la\|\na_y(u-U_0)\|_{L_s^2\left(0,T;L_y^2(\R^3)\right)}^2.\]
Therefore,
\EQ{\label{eq_1_la^2}
v-e^{t\De} v_0\in L^\infty(1,\la^2;L^2(\R^3))\cap L^2(1,\la^2;H^1(\R^3)).
}

Note that $v-e^{t\De} v_0$ is $\la$-DSS because $u-U_0$ is $T$-periodic, where $T=\log(\la)$. For $t>0$, $\la^{-2k}\le t< \la^{-2k+2}$ for some $k\in\ZZ$ so $1\le\la^{2k}t<\la^2$. Thus
\EQ{\label{eq_4.1_1}
\|v(t)-e^{t\De}v_0\|_{L^2(\R^3)}^2=&~\la^{-1}\int_{\R^3}\left|(v-e^{t\De}v_0)(x,\la^2t)\right|^2dx\\
=&~\cdots\\
=&~\la^{-k}\int_{\R^3}\left|(v-e^{t\De}v_0)(x,\la^{2k}t)\right|^2dx\\
\le&~t^{1/2}\sup_{1\le\tau\le\la^2}\|v(\tau)-e^{\tau\De}v_0\|_{L^2(\R^3)}^2.
}
Moreover, 
\EQ{\label{eq_4.1_2}
\int_{\la^{-2k}}^{\la^{-2k+2}}\int_{\R^3}\left|\na(v(t)-e^{t\De}v_0)\right|^2dxdt=&~\la^{-1}\int_{\la^{-2k+2}}^{\la^{-2k+4}}\int_{\R^3}\left|\na(v(t)-e^{t\De}v_0)\right|^2dxdt\\
=&~\cdots\\
=&~\la^{-k}\int_1^{\la^2}\int_{\R^3}\left|\na(v(t)-e^{t\De}v_0)\right|^2dxdt
}
implies that
\EQ{\label{eq_4.1_3}
\int_0^{\la^2}\int_{\R^3}\left|\na(v(t)-e^{t\De}v_0)\right|^2dxdt=&~\sum_{k=0}^\infty\int_{\la^{-2k}}^{\la^{-2k+2}}\int_{\R^3}\left|\na(v(t)-e^{t\De}v_0)\right|^2dxdt\\
=&~\left(\sum_{k=0}^\infty\la^{-k}\right)\int_1^{\la^2}\int_{\R^3}\left|\na(v(t)-e^{t\De}v_0)\right|^2dxdt.
}
Therefore, we see from \eqref{eq_4.1_1} and \eqref{eq_4.1_3} that 
\EQ{\label{eq_4.1_4}
v-e^{t\De}v_0\in L^\infty(0,\la^2;L^2(\R^3))\cap L^2(0,\la^2;H^1(\R^3)).
}

We first prove that $v$ has locally finite energy and enstrophy. In view of Remark 3.2 in \cite{MR673830}, we have
\EQN{
&~~~\sup_{x_0\in\R^3}\int_{B_R(x_0)}|e^{t\De}v_0(x)|^2dx\\&=\sup_{x_0\in\R^3}\int_{B_1(x_0)}|e^{t\De}v_0(R(\tilde{x}-x_0)+x_0)|^2R^3\,d\tilde{x}\\
&=\sup_{x_0\in\R^3}\int_{B_1(x_0)}|e^{\frac{t}{R^2}\De}\widetilde{v_0}(\tilde{x})|^2R^3\,d\tilde{x}, \text{ where }\widetilde{v_0}(\tilde{x})=v_0(R(\tilde{x}-x_0)+x_0),\\
&=R^3\,\|e^{\frac{t}{R^2}\De}\widetilde{v_0}\|_{L^2_{\textup{uloc}}}^2\\
&\lesssim R^3\,\|\widetilde{v_0}\|_{L^2_{\textup{uloc}}}^2\ \ \ (\text{by Remark 3.2 in \cite{MR673830}})\\
&=\sup_{x_0\in\R^3}\int_{B_R(x_0)}|v_0(x)|^2dx\\
&=\la^k\sup_{x_0\in\R^3}\int_{B_{\la^{-k}R}(\la^{-k}x_0)}|v_0(x)|^2dx, \text{ where $\la^{k-1}\le R<\la^k$ for some $k$, }(\text{since $v_0$ is $\la$-DSS})\\
&\le\la^k\sup_{x_0\in\R^3}\int_{B_1(\la^{-k}x_0)}|v_0(x)|^2dx\\
&=\la^k\|v_0\|_{L^2_{\textup{uloc}}}^2.
}
Combining this result with \eqref{eq_4.1_1}, we actually have
\EQ{\label{v_energy}
&~~~~\underset{0\le t<R^2}{\textup{esssup}}\,\sup_{x_0\in\R^3}\int_{B_R(x_0)}|v(x,t)|^2dx\\
&\le2\left(\underset{0\le t<R^2}{\textup{esssup}}\,\sup_{x_0\in\R^3}\int_{B_R(x_0)}|v(x,t)-e^{t\De}v_0|^2dx+\underset{0\le t<R^2}{\textup{esssup}}\,\sup_{x_0\in\R^3}\int_{B_R(x_0)}|e^{t\De}v_0|^2dx\right)\\
&\lesssim2\left(R\sup_{1\le\tau\le\la^2}\|v(\tau)-e^{\tau\De}v_0\|_{L^2(\R^3)}+R\la\|v_0\|_{\textup{uloc}}^2\right)<\infty.
}
Likewise, since 
\EQN{
\sup_{x_0\in\R^3}\int_{B_R(x_0)}|\na_x(e^{t\De}v_0(x))|^2dxdt=&~\sup_{x_0\in\R^3}\int_{B_1(x_0)}\left|R^{-1}\na_{\tilde{x}}(e^{t\De}v_0)(R(\tilde{x}-x_0)+x_0)\right|^2R^3\,d\tilde{x}\\
=&~\sup_{x_0\in\R^3}\int_{B_1(x_0)}\left|\na_{\tilde{x}}(e^{\frac{t}{R^2}\De}\widetilde{v_0})(\tilde{x})\right|^2R\,d\tilde{x}\\
\lesssim&~\frac{R}{\left(\frac{t}{R^2}\right)^{\frac12}}\,\|\widetilde{v_0}\|_{L^2_{\textup{uloc}}}^2\ \ \ (\text{by Remark 3.2 in \cite{MR673830}})\\
\le&~\frac{\la}{t^{\frac12}}\,\|v_0\|_{L^2_{\textup{uloc}}}^2,
}
it follows from \eqref{eq_4.1_2} that
\EQ{\label{v_enstrophy}
&~~~~\sup_{x_0\in\R^3}\int_0^{R^2}\int_{B_R(x_0)}|\na v(x,t)|^2dxdt\\
&\le2\left(\sup_{x_0\in\R^3}\int_0^{R^2}\int_{B_R(x_0)}|\na(v-e^{t\De}v_0)|^2dxdt+\sup_{x_0\in\R^3}\int_0^{R^2}\int_{B_R(x_0)}|\na(e^{t\De}v_0)|^2dxdt\right)\\
&\le2\left(\sum_{m=0}^\infty\la^{-(k+m)}\int_1^{\la^2}\int_{\R^3}\left|\na(v(t)-e^{t\De}v_0)\right|^2dxdt+2R\la\|v_0\|_{L^2_{\textup{uloc}}}^2\right)\\
&\le2\left(\frac{R\la}{\la-1}\int_1^{\la^2}\int_{\R^3}\left|\na(v(t)-e^{t\De}v_0)\right|^2dxdt+2R\la\|v_0\|_{L^2_{\textup{uloc}}}^2\right)<\infty,
}
where $k$ is some integer so that $\la^{k-1}\le R<\la^k$. The same conclusion of \eqref{v_energy} and \eqref{v_enstrophy} can be drawn for $b(t)-e^{t\De}b_0$. This proves \eqref{lfee_mhd}.

Secondly, we prove the convergence to initial data. Let $K$ be a compact subset of $\R^3$. We split $\|v(t)-v_0\|_{L^2_{\textup{loc}}}$ into two parts: $\|v(t)-e^{t\De}v_0\|_{L^2_{\textup{loc}}}$ and $\|e^{t\De}v_0-v_0\|_{L^2_{\textup{loc}}}$. The first part is controlled by \eqref{eq_4.1_1} as 
\EQ{\label{conv_initial_1}
\|v(t)-e^{t\De}v_0\|_{L^2(K)}\lesssim t^{1/4}\to0\ \text{ as $t\to0^+$}.
}
For the second part, we use the fact that $e^{t\De}v_0\to v_0$ in $L^2_{-3/2}$ as $t\to 0^+$ mentioned in the Remark 2.3 of \cite{MR1179482}. Moreover, we have the embeddings $L^3_w\subset M^{2,1}\subset L^2_{-3/2}\subset L^2_{\textup{loc}}$ (see \ref{sec_append} Appendix). Hence $e^{t\De}v_0\to v_0$ in $L^2_{-3/2}$ as $t\to 0^+$ implies
\EQ{\label{conv_initial_2}
e^{t\De}v_0\to v_0\ \text{ in }L^2_{\textup{loc}}\ \text{ as $t\to0^+$}.
}
Therefore, combining \eqref{conv_initial_1} and \eqref{conv_initial_2}, we have \EQ{\label{conv_initial_v}
v\to v_0\ \text{ in }L^2_{\textup{loc}}\ \text{ as $t\to0^+$}.
}
The same convergence \eqref{conv_initial_v} is true for $b$. This establishes the convergence to initial data.

Next, we prove the decay at spatial infinity. Fix any $R>0$. We split $v$ into two parts: $v-e^{t\De}v_0$ and $e^{t\De}v_0$. For the first part, $v-e^{t\De}v_0\in L^2(0,R^2;L^2(\R^3))$ since \[\int_0^{R^2}\int_{\R^3}|(v-e^{t\De}v_0)(x,t)|^2dxdt\le\int_0^{R^2}t^{1/2}\underset{1\le\tau\le\la^2}{\sup}\|(v-e^{\tau\De}v_0)(x,\tau)\|_{L^2(\R^3)}^2dt<\infty\] by \eqref{eq_4.1_1}. The dominated convergence theorem then implies 
\[\int_0^{R^2}\int_{B_R(x_0)}|(v-e^{t\De}v_0)(x,t)|^2dxdt=\int_0^{R^2}\int_{\R^3}|(v-e^{t\De}v_0)(x,t)|^21_{B_R(x_0)}(x)\,dxdt\to0\]
as $|x_0|\to\infty$. For the second part, since $v_0$ is $
\la$-DSS, $e^{t\De}v_0$ is also $\la$-DSS and $U_0$ is periodic in $s$ with the period $T=\log(\la)$. So \eqref{eq_1_la^2} and \eqref{eq_4.1_1} also hold for $e^{t\De}v_0$. In the same manner above, we can show \[\int_0^{R^2}\int_{B_R(x_0)}|e^{t\De}v_0(x)|^2dxdt\to0\] as $|x_0|\to\infty$. Since the same proof works for $b$, we can conclude that \eqref{dasi_mhd} holds.

Finally, the local energy inequality \eqref{lei_mhd} for \eqref{MHD} follows from the local energy inequality \eqref{lei_mhd_leray} for \eqref{leray_mhd}.
\end{proof}

\subsection{Discretely self-similar solutions to the viscoelastic Navier-Stokes equations with damping}
\begin{proof}[Proof of \thref{thm_1.2_vNSEd}]
Let $U_0(y,s)=\sqrt{2t}(e^{t\De}v_0)(x)$ and $(G_m)_0=\sqrt{2t}(e^{t\De}(f_0)_m)(x),\,m=1,2,3,$ where $(f_0)_m$ is the $m$-th column of ${\bf F}_0$. By \thref{lem_3.4}, $U_0$ and $(G_m)_0,\,m=1,2,3,$ all satisfy \thref{assum_2.1} with $T=\log\la$ and $q=10/3$. Let $(u,g_1,g_2,g_3,p)$ be the $T$-periodic weak solution derived in \thref{thm_2.4_vNSEd}. Let $v(x,t)=u(y,s)/\sqrt{2t},\,{\bf F}(x,t)={\bf G}(y,s)/\sqrt{2t}$ and $\pi(x,t)=p(y,s)/2t$ where ${\bf G}=(g_1,g_2,g_3)$ and $x,t,y,s$ satisfy \eqref{xtys}. We skip the rest of the proof as it is essentially the same as that in Sect. \ref{sect_3.1}.
\end{proof}

\section{Self-Similar Solutions}
In this section, we prove \thref{thm_1.3_mhd} and \thref{thm_1.3_vNSEd}.
\subsection{Self-similar solutions to the MHD equations}\label{sect_4.1}
\begin{proof}[Proof of \thref{thm_1.3_mhd}]
Let $U_0$ and $A_0$ be defined as in Sect. \ref{sect_3.1}. Since $v_0$ and $b_0$ are $(-1)$-homogeneous, \[U_0(y)=2^{3/2}(4\pi)^{-3/2}\int_{\R^3}e^{-|y-z|^2/2}v_0(z)dz\text{ and }A_0(y)=2^{3/2}(4\pi)^{-3/2}\int_{\R^3}e^{-|y-z|^2/2}b_0(z)dz\] are independent of $s$. By \thref{lem_3.4}, $U_0$ and $A_0$ both satisfy \thref{assum_2.1} for any $q\in(3,\infty]$ because $v_0$ and $b_0$ are $\la$-DSS for all $\la>1$. Let $W$ and $D$ be defined as in \eqref{W_def} and \eqref{D_def}, respectively. Then $W$ and $D$ are independent of $s$. Furthermore, according to \thref{lem_2.5}, $W$ and $D$ satisfy the estimates \eqref{eq_2.8}-\eqref{eq_2.10} with $q\in(3,\infty]$. Our goal is to solve the following variational form of the stationary Leray system for the MHD equations
\EQ{\label{eq_5.3_mhd}
\setlength\arraycolsep{1.5pt}\def\arraystretch{1.2}
\left\{\begin{array}{ll}
-(\na u,\na f)+(u+y\cdot\na u-u\cdot\na u+a\cdot\na a,f)&=0\\
-(\na a,\na f)+(a+y\cdot\na a-u\cdot\na a+a\cdot\na u,f)&=0,
\end{array}
\right.
}
for all $f\in\mathcal{V}$. Similar to the proof of \thref{thm_2.4_mhd}, we are looking for a solution of the form $u=W+U$ and $a=D+A$ and using Galerkin method to achieve this. Note that $(U,A)$ satisfies the perturbed stationary Leray system for the MHD equations, which has the weak formulation as 
\EQ{\label{eq_5.4_mhd}
\setlength\arraycolsep{1.5pt}\def\arraystretch{1.2}
\left\{\begin{array}{ll}
&-(\na U,\na f)+(U+y\cdot\na U,f)-(U\cdot\na U-A\cdot\na A,f)\\
&~~~~~~~~~~~~~~~~~~~~~~=(W\cdot\na U+U\cdot\na W-D\cdot\na A-A\cdot\na D,f)+\left<\mathcal{R}_1(W,D),f\right>\\
&-(\na A,\na f)+(A+y\cdot\na A,f)-(U\cdot\na A-A\cdot\na U,f)\\
&~~~~~~~~~~~~~~~~~~~~~~=(W\cdot\na A+U\cdot\na D-D\cdot\na U-A\cdot\na W,f)+\left<\mathcal{R}_2(W,D),f\right>,
\end{array}
\right.
}
for all $f\in\mathcal{V}$, where $\mathcal{R}_1$ and $\mathcal{R}_2$ are the same as in \eqref{eq_R1_R2}. Let $\{h_k\}_{k\in\NN}\subset\mathcal{V}$ be an orthonormal basis of $H$. For a fixed $k$, we look for an approximation solution of the form $U_k(y)=\sum_{i=1}^k\mu_{ki}h_i(y),\,A_k(y)\sum_{i=1}^k\al_{ki}h_i(y)$. Plugging them into the weak formulation, we get the following algebraic system: 
\begin{equation}\label{eq_5.5_mhd}
\setlength\arraycolsep{1.5pt}\def\arraystretch{1.2}
\left\{\begin{array}{ll}
\underset{i=1}{\overset{k}\sum}\mathscr{A}_{ij}\mu_{ki}+\underset{i=1}{\overset{k}\sum}\mathscr{B}_{ij}\al_{ki}+\underset{i,l=1}{\overset{k}\sum}\mathscr{C}_{ilj}\mu_{ki}\mu_{kl}-\underset{i,l=1}{\overset{k}\sum}\mathscr{C}_{ilj}\al_{ki}\al_{kl}+\mathscr{D}_j&=0\\
\underset{i=1}{\overset{k}\sum}\mathscr{E}_{ij}\mu_{ki}+\underset{i=1}{\overset{k}\sum}\mathscr{F}_{ij}\al_{ki}+\underset{i,l=1}{\overset{k}\sum}\mathscr{G}_{ilj}\mu_{ki}\al_{kl}+\mathscr{H}_j&=0,
\end{array}\right.
\end{equation}
for $j=1,\cdots,k$, where $\mathscr{A}_{ij},\,\mathscr{B}_{ij},\,\mathscr{D}_j,\,\mathscr{E}_{ij},\,\mathscr{F}_{ij},\,\mathscr{H}_j$ are the same as those in \eqref{galerkin_ode_coeff_mhd}, and 
\EQ{
\mathscr{C}_{ilj}&=-(h_i\cdot\na h_l,h_j),\\
\mathscr{G}_{ilj}&=-(h_i\cdot\na h_l,h_j)+(h_l\cdot\na h_i,h_j).
}
Let $P:\R^{2k}\to\R^{2k}$ be defined by
\EQN{
&(P(\mu_{k1},\cdots,\mu_{kk},\al_{k1},\cdots,\al_{kk}))_j\\
&~~~~=\left\{\begin{array}{ll}\underset{i=1}{\overset{k}\sum}\mathscr{A}_{ij}\mu_{ki}+\underset{i=1}{\overset{k}\sum}\mathscr{B}_{ij}\al_{ki}+\underset{i,l=1}{\overset{k}\sum}\mathscr{C}_{ilj}\mu_{ki}\mu_{kl}-\underset{i,l=1}{\overset{k}\sum}\mathscr{C}_{ilj}\al_{ki}\al_{kl}+\mathscr{D}_j,&\ j=1,\cdots,k,\\
\underset{i=1}{\overset{k}\sum}\mathscr{E}_{i(j-k)}\mu_{ki}+\underset{i=1}{\overset{k}\sum}\mathscr{F}_{i(j-k)}\al_{ki}+\underset{i,l=1}{\overset{k}\sum}\mathscr{G}_{il(j-k)}\mu_{ki}\al_{kl}+\mathscr{H}_{j-k},&\ j=k+1,\cdots,2k.\end{array}\right.
}
From similar estimates as in \eqref{etm_2.28_mhd} and \eqref{etm_2.29_mhd}, we have  that
\EQ{\label{eq_5.6_mhd}
&P(\mu_{k1},\cdots,\mu_{kk},\al_{k1},\cdots,\al_{kk})\cdot(\mu_{k1},\cdots,\mu_{kk},\al_{k1},\cdots,\al_{kk})\\
=&~-\frac12\left(\|U_k\|_{L^2}^2+\|A_k\|_{L^2}^2\right)-\left(\|\na U_k\|_{L^2}^2+\|\na A_k\|_{L^2}^2\right)\\
&~-(U_k\cdot\na W-D\cdot\na A_k-A_k\cdot\na D,U_k)-(U_k\cdot\na D-D\cdot\na U_k-A_k\cdot\na W,A_k)\\
&~-\left<\mathcal{R}_1(W,D),U_k\right>-\left<\mathcal{R}_2(W,D),A_k\right>\\
\le&~-\frac12\left(\|U_k\|_{L^2}^2+\|A_k\|_{L^2}^2\right)-\left(\|\na U_k\|_{L^2}^2+\|\na A_k\|_{L^2}^2\right)+\frac38\left(\|U_k\|_{H^1}^2+\|A_k\|_{H^1}^2\right)\\
&~+C_2+\frac3{32}\left(\|U_k\|_{H^1}^2+\|A_k\|_{H^1}^2\right)\\
=&~-\frac1{32}\left(\|U_k\|_{L^2}^2+\|A_k\|_{L^2}^2\right)-\frac{17}{32}\left(\|\na U_k\|_{L^2}^2+\|\na A_k\|_{L^2}^2\right)+C_2\\
\le&~-\frac1{32}\left|(\mu_{k1},\cdots,\mu_{kk},\al_{k1},\cdots,\al_{kk})\right|^2+C_2\\
<&~0,
}
if $\left|(\mu_{k1},\cdots,\mu_{kk},\al_{k1},\cdots,\al_{kk})\right|=8\sqrt{C_2}=:\rho$. Note that $C_2$ is independent of $k$. Thus, we obtain a point $(\mu_{k1},\cdots,\mu_{kk},\al_{k1},\cdots,\al_{kk})\in B_\rho^{2k}$ such that $P(\mu_{k1},\cdots,\mu_{kk},\al_{k1},\cdots,\al_{kk})=0$ by Brouwer's fixed point theorem. Then $U_k(y)=\sum_{i=1}^k\mu_{ki}h_i(y),\,A_k(y)\sum_{i=1}^k\al_{ki}h_i(y)$ is our approximation solution of \eqref{eq_5.4_mhd} with a priori bound 
\[\left(\|U_k\|_{L^2}^2+\|A_k\|_{L^2}^2\right)+17\left(\|\na U_k\|_{L^2}^2+\|]\na A_k\|_{L^2}^2\right)\le32\,C_2.\]
Therefore, we have, up to a subsequence, the following convergences
\EQ{
&U_k\rightharpoonup U,\ A_k\rightharpoonup A\ \text{ weakly in $H^1(\R^3)$},\\
&U_k\to U,\ A_k\to A\ \text{ strongly in $L^2(K)$ for all compact sets $K\subset\R^3$}.
}
So we derive a solution $(U,A)$ to \eqref{eq_5.4_mhd} with $U,A\in H^1(\R^3)$. Then $(u,a)$, where $u=U+W$ and $a=A+D$, is a solution to \eqref{eq_5.3_mhd}. Note that $u,a\in H^1_{\textup{loc}}\cap L^q$ for all $3<q\le 6$ since $U,A\in H^1\subset L^q$ for $q\le6$ and $W,D\in L^q\cap L^4\cap C^\infty_{\text{loc}}$ for $q>3$. 

Regarding the pressure, we define \[p=\sum_{i,j=1}^3R_iR_j(u_iu_j-a_ia_j),\] where $R_i$ stands for the Riesz transforms. Then $(u,a,p)$ satisfies the stationary Leray system for the MHD equations \eqref{eq_1.7_mhd} in the sense of distributions. Moreover, Calderon-Zygmund estimates gives the following a priori bound for $p$: for $3<q\le6$ \[\|p\|_{L^{q/2}(\R^3)}\le C\|u\|_{L^q(\R^3)}^2.\]

Recovering $(v,b,\pi)$ from $(u,a,p)$ by the relation \eqref{eq_1.6_mhd}, we obtain a self-similar weak solution of \eqref{MHD} (see \cite[pp.33-34]{MR1643650}). It remains to show that $(v,b,\pi)$ is a local Leray solution of \eqref{MHD}.

Recall that $(U,p)$ is a solution of the stationary Stokes system with the force 
\[\mathcal{G}_1=U+y\cdot\na U-(U\cdot\na U-A\cdot\na A)-W\cdot\na U-U\cdot\na W+D\cdot\na A+A\cdot\na D-\mathcal{L}W-W\cdot\na W+D\cdot\na D.\]
Applying the regularity result in \cite[Proposition 1.2.2]{MR0609732} on compact subsets of $\R^3$, $u$ and $p$ are actually smooth. Additionally, $A$ is a solution of the Poisson equation with the right hand side 
\[\mathcal{G}_2=A+y\cdot\na A-(U\cdot\na A-A\cdot\na U)-W\cdot\na A-U\cdot\na D+D\cdot\na U+A\cdot\na W-\mathcal{L}D-W\cdot\na D+D\cdot\na W.\]
A standard elliptic regularity result leads to the smoothness for $A$ on compact subsets of $\R^3$. Thus, $u,a$ and $p$ inherit the smoothness from $U,W,A$ and $D$. Therefore, from the self-similarity of $v,b$ and $\pi$, they are smooth in both spatial and time variables. Consequently, the local energy inequality \eqref{lei_mhd} can be achieved via integrating by parts. The rest of conditions from \thref{def_loc_leray_mhd} and the estimates of the distance between the solution $(v,b)$ and the background $(e^{t\De}v_0,e^{t\De}b_0)$ can be verified using the same approach as in Sect. \ref{sect_3.1}.
\end{proof}

\subsection{Self-similar solutions to the viscoelastic Navier-Stokes equations with damping}
\begin{proof}[Proof of \thref{thm_1.3_vNSEd}]
The proof is basically the same as in Sect. \ref{sect_4.1}. It is worth noting that in \eqref{eq_5.6_mhd} we use the estimates \eqref{etm_2.28_mhd} and \eqref{etm_2.29_mhd} obtained by applying \thref{lem_2.5} with $\de=\frac14$; while here we acheive \eqref{eq_5.6_mhd} from estimates \eqref{etm_2.28_vNSEd} and \eqref{etm_2.29_vNSEd} by applying the same lemma but with the parameter $\de=\frac18$. The details of verification are left to the reader.
\end{proof}

\section{Appendix}\label{sec_append}
In this appendix, we prove the three inclusions $L^3_w\subset M^{2,1}\subset L^2_{-3/2}\subset L^2_{\textup{loc}}$. To begin with, the first inclusion can be shown by the inequality
\EQN{
r^{-1}\int_{B_r(x_0)}|f(x)|^2dx=&~r^{-1}\int_{B_r(x_0)}\int_0^{|f(x)|}2\al\,d\al dx\\
=&~r^{-1}\int_{B_r(x_0)}\int_0^\infty2\al1_{|f|>\al}(x)\,d\al dx\\
=&~r^{-1}\int_0^\infty2\al|\{|f|>\al\}\cap B_r(x_0)|d\al\\
=&~r^{-1}\int_0^{r^{-1}}2\al|\{|f|>\al\}\cap B_r(x_0)|d\al+r^{-1}\int_{r^{-1}}^\infty2\al|\{|f|>\al\}\cap B_r(x_0)|d\al\\
\le&~r^{-1}\int_0^{r^{-1}}2\al|B_r(x_0)|d\al+r^{-1}\int_{r^{-1}}^\infty2\al|\{|f|>\al\}|d\al\\
\le&~r^{-1}|B_r(x_0)|r^{-2}+r^{-1}\int_{r^{-1}}^\infty2\al\|f\|_{L^3_w}^3\al^{-3}d\al\\
\lesssim&~1+\|f\|_{L^3_w}^3.
} 
Next, the second inclusion is valid as
\EQN{
\int_{\R^3}\frac{|f(x)|^2}{(1+|x|)^3}\,dx=&~\int_{|x|<1}\frac{|f(x)|^2}{(1+|x|)^3}\,dx+\sum_{k=0}^\infty\int_{2^k\le|x|<2^{k+1}}\frac{|f(x)|^2}{(1+|x|)^3}\,dx\\
\le&~\int_{B_1(0)}|f(x)|^2dx+\sum_{k=0}^\infty\frac1{(1+2^k)^3}\int_{2^k\le|x|<2^{k+1}}|f(x)|^2dx\\
\le&~\|f\|_{M^{2,1}}^2+\sum_{k=0}^\infty\frac1{(1+2^k)^3}\,2^{k+1}\|f\|_{M^{2,1}}^2\\
\lesssim&~\|f\|_{M^{2,1}}^2.
}
Finally, the third inclusion holds since
\EQN{
\int_{|x|\le M}|f(x)|^2dx\le(1+M)^3\int_{\R^3}\frac{|f(x)|^2}{(1+|x|)^3}\,dx=(1+M)^3\|f\|_{L^2_{-3/2}}^2.
}

\section*{Acknowledgments}
The research was partially supported by FYF (\#6456) of Graduate and Postdoctoral Studies, University of British Columbia (BC). The author would like to express his fully gratitude to Tai-Peng Tsai for kindly discussion. Also, he thanks Anyi Bao for her proofreading.

 \bibliography{local_leray_mhd}
\bibliographystyle{abbrv}
\nocite{*}



 
\end{document}